\documentclass[reqno,12pt,letterpaper]{amsart}
\usepackage[proof]{zlmacros}
\usepackage{mathtools}

\newcommand{\ep}{\epsilon}

\newcommand{\abs}[1]{{\lvert{#1}\rvert}}
\newcommand{\norm}[1]{{\lVert{#1}\rVert}}
\newcommand{\ang}[1]{{\langle{#1}\rangle}}
\newcommand{\RR}{\mathbb{R}}

\newcommand{\NN}{\mathbb{N}}

\newcommand{\CC}{\mathbb{C}}
\newcommand{\F}{\mathcal{F}}
\newcommand{\pa}{\partial}
\newcommand{\md}{\mathrm{d}}
\newcommand{\mb}{\mathrm{b}}

\renewcommand{\Im}{\operatorname{Im}}
\renewcommand{\Re}{\operatorname{Re}}

\title{Evolution of internal waves in a 2D subcritical channel}
\author{Zhenhao Li, Jian Wang, Jared Wunsch}
\date{}

\begin{document}

\begin{abstract}
    We study the long time evolution of internal waves in two dimensional subcritical channels with flat horizontal ends. We show the leading profiles of solutions are the outgoing solutions to the stationary equations. This is done by showing a limiting absorption principle and the differentiability of the spectral measure for a zeroth order operator.
\end{abstract}

\maketitle

\section{Introduction}
Let $\Omega \subset \R^2$ be an open domain. We consider the 2D linear internal waves equation forced periodically by a profile $f \in \CIc(\Omega; \R)$ at the time frequency $\lambda\in (0,1)$, which is given by 
\begin{equation}\label{eq:IW}
    (\partial_t^2 \Delta + \partial_{x_2}^2) u = f(x) \cos \lambda t, \quad u|_{t = 0} = \partial_t u|_{t = 0} = 0, \quad u|_{\partial \Omega} = 0.
\end{equation}
In this paper, we study the equation in an unbounded channel with flat horizontal ends. More precisely, we consider domains of the form
\begin{equation*}
    \Omega \coloneqq \{x \in \R^2: G(x_1) - \pi < x_2 < 0\} \quad \text{with} \quad G \in \CIc(\R; \R), \  G<\pi.
\end{equation*}
As in~\cite{LiWaWu:24}, we further make the assumption that $\Omega$ is $\lambda$-subcritical, meaning
\begin{equation}\label{eq:subcrit}
    \max_{x_1 \in \R} |G'(x_1)| < \frac{\sqrt{1 - \lambda^2}}{\lambda}.
\end{equation}
See \S\ref{sec:dynamics} for a detailed discussion of this condition.

Our main result describes the long time evolution for solutions to \eqref{eq:IW}.
\begin{theorem}\label{thm:tide}
    Suppose $\Omega$ is $\lambda$-subcritical and let $f \in \CIc(\Omega; \R)$. Then the solution to~\eqref{eq:IW} has the decomposition
    \[u(t) = \Re(e^{i \lambda t} u^+) + r(t) + \mathcal E(t)\]
    where $u^+ \in \bar H^{s, \beta}_0(\Omega)$ for any $s \ge 1$ and $\beta < -1/2$ (%see~\eqref{eq:weighted_sob} 
    see~\S \ref{sssection:sobolev} for the definition of the weighted Sobolev space). Furthermore, we have a remainder uniformly bounded in energy space 
    \[\|r(t)\|_{H^1_0(\Omega)} \le C\]
    for some $C>0$ and all $t$, as well as an error term 
    \[\|\mathcal E(t)\|_{\bar H^{s, \beta}(\Omega)} \to 0\]
    as $t \to \infty$. 
\end{theorem}

We approach the evolution problem using spectral methods, similar to the setups in \cite{DyWaZw:21, Li:23, Li:24}. Since $\Omega$ is bounded in the $x_2$ direction, the Dirichlet Laplacian has an inverse $\Delta_\Omega^{-1} :H^{-1}(\Omega) \to H^1_0(\Omega)$. Then we can replace the internal waves equation~\eqref{eq:IW} by 
\begin{equation}\label{eq:IWP}
    (\partial_t^2 + P) w = f \cos \lambda t, \quad w|_{t = 0} = \partial_t w|_{t = 0} = 0, \quad f \in \CIc(\Omega; \R), \quad u = \Delta_{\Omega}^{-1} w, 
\end{equation}
where
\begin{equation}\label{eq:0pseudo}
    P \coloneqq \partial_{x_2}^2 \Delta_{\Omega}^{-1}: H^{-1}(\Omega)\to H^{-1}(\Omega).
\end{equation}
Therefore, studying the evolution problem becomes a problem of understanding the spectrum of $P$. Equip the Hilbert space $H^{-1}(\Omega)$ with the norm 
\[\langle u, w \rangle_{H^{-1}(\Omega)} \coloneqq \langle \nabla \Delta_\Omega^{-1} u, \nabla \Delta_{\Omega}^{-1} w \rangle_{L^2(\Omega)}.\]
Then it can be shown that $P: H^{-1}(\Omega) \to H^{-1}(\Omega)$ is non-negative, bounded, and self-adjoint with $\mathrm{Spec}(P) = [0, 1]$. The solution to the equation~\eqref{eq:IWP} can then be expressed using the functional calculus of $P$:
\begin{equation}\label{eq:functional_sol}
    \begin{gathered}
        w(t) = \Re\big(e^{i \lambda t} \mathbf W_{t, \lambda} (P) f\big), \\
        \mathbf W_{t, \lambda} (z) \coloneqq \int_0^t \frac{\sin(s \sqrt{z})}{\sqrt{z}} e^{-i \lambda s}\, \md s = \sum_\pm \frac{1 - e^{-i t(\lambda \pm \sqrt{z})}}{2 \sqrt{z}(\sqrt{z} \pm \lambda)}.
    \end{gathered}
\end{equation}
As $t \to \infty$, we see that 
\begin{equation}\label{Wlim}
\mathbf W_{t, \lambda}(z) \to (z - \lambda^2 + i0)^{-1} \quad \text{in} \quad \mathcal D'_z((0, \infty)).
\end{equation}
See \cite[\S 1]{DyWaZw:21}. Therefore, our goal is to develop a \emph{limiting absorption principle}, that is, to understand the limit 
\begin{equation}\label{limitingresolvent}\lim_{\epsilon \to 0} (P - (\lambda\pm i \epsilon)^2)^{-1} f. \end{equation}
This limit is needed both in Stone's formula to understand the spectral measure, and to give the limiting profile $\lim_{t \to \infty} \mathbf W_{t, \lambda}(P) f$ of the internal waves equation. In terms of the spectral measure, we have the following theorem. 
\begin{theorem}\label{thm:spectral}
    Let $\mathcal J \subset [0, 1]$ be an open interval such that $\Omega$ is $\lambda$-subcritical with respect to every $\lambda \in \mathcal J$. Let $u_{\lambda \pm i0}$ be the unique incoming (or outgoing) solutions to the stationary internal wave equation (see \S\ref{sec:dynamics} or \cite{LiWaWu:24} for details). Then 
    \[\lim_{\epsilon \to 0+} \Delta_{\Omega}^{-1} (P - \lambda^2 \mp i\epsilon)^{-1} f = u_{\lambda \pm i0} \quad \text{in $\bar H^{s, \beta}_0(\Omega)$}\]
    for every $s \ge 1$ and $\beta < -1/2$. In particular, $\Spec(P)$ is purely absolutely continuous in $\mathcal J^2 = \{\lambda^2 : \lambda \in \mathcal J\}$.
\end{theorem}

\begin{Remark}
    In fact, we will prove a more precise version of this result in Propositions~\ref{prop:LAP} and~\ref{prop:LAP_diff}, where we will show that the spectral measure is differentiable. The more precise version will then imply Theorem~\ref{thm:tide} via the functional solution~\eqref{eq:functional_sol}. In particular, we will see that the limiting profile in Theorem~\ref{thm:tide} is given by $u^+ = u_{\lambda - i0}$. In other words, up to errors bounded in energy, the internal wave equation will converge to the outgoing solution to the stationary internal wave equation.
\end{Remark}

\begin{figure}[t]
    \centering
    \includegraphics{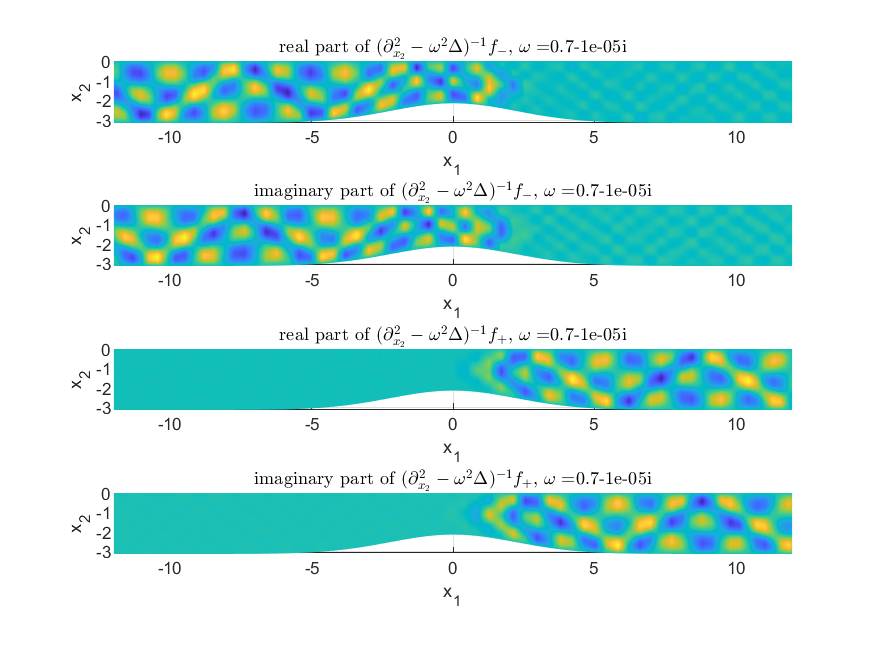}
    \caption{Numerical computation of $u_\omega=P(\omega)^{-1}f_{\pm}$ (see \eqref{eq:P(w)u} for the definition) with $\omega=0.7-10^{-5}i$, $f_{\pm}(x) = a(x)e^{\pm 5ix_1}$ and $a$ is a Gaussian function in $y$ coordinates (introduced in Remark below) centered at $y=(0.1,0)$. }
    \label{fig:sub_osc}
\end{figure}

\begin{Remark}
    Let us describe the numerical method used to compute Figure \ref{fig:sub_osc}. The goal is to approximate $u_\omega$ in $\Omega_L\coloneqq\Omega\cap \{|x_1|<L\}$ for given $L>0$. We first introduce change of coordinates
    \[ \Omega_L\ni (x_1,x_2)\mapsto (y_1,y_2)\coloneqq \left( \frac{x_1}{L}, 1+\frac{2x_2}{\pi-G(x_1)} \right)\in (-1,1)\times (-1,1). \]
    In order to approximate outgoing solutions in the {\em unbounded} domain $\Omega$ by solutions in a {\em bounded} domain with Dirichlet boundary condition, we introduce the complex scaling in the $y_1$ variable (see \cite[\S 2.7]{DyZw:19} for an introduction of complex scaling in the context of Schr\"odinger operator on the real line). More precisely, we replace $(-1,1)_{y_1}$ by a complex curve 
    \[ \Gamma_{\tau}\coloneqq \{ \gamma_\tau(y_1) : y_1\in (-1,1)\}, \ \gamma_\tau(y_1)\coloneqq y_1(1+i\tau \rho(y_1)), \ \tau>0 \]
    with $\rho\in C^{\infty}([-1,1];\RR)$ and for some $\delta>0$, in $y$-coordinates,
    \[\begin{gathered}  
    \supp G' \text{ and } y_1\text{-projection of } \supp f \text{ are contained in } [-1+\delta,1-\delta], \\
    \rho|_{[-1+\delta,1-\delta]}= 0, \ 
    \rho= 1 \text{ near } y_1=\pm 1.
    \end{gathered}\]
    Let $P_y(\omega)$ be $P(\omega)$ in $y$-coordinates. Then $P_y(\omega)$ can be holomorphically extended to an operator $P_{\Gamma}(\omega)$ on $\Gamma_\tau\times(-1,1)$. We then approximate $u_{\omega}$ in $\Omega_{(1-\delta)L}$ by the solution $u_{\omega, \Gamma}$ to 
    \begin{equation}\label{eq:complex}
    P_{\Gamma}(\omega) u_{\omega,\Gamma}=f, \ u_{\omega,\Gamma}|_{\partial(\Gamma_\tau\times(-1,1))}=0. 
    \end{equation}
    We compute $P_\Gamma(\omega)$ explicitly 
    \[\begin{split}
    P_{\Gamma}(\omega) =  -\omega^2 & \left[ \tfrac{1}{L^2}\partial_{z_1}^2+\tfrac{1}{L}\left( \tfrac{G'(Lz_1)}{\pi-G(Lz_1)} \right)'_{z_1} (z_2-1)\partial_{z_2}+\tfrac{2}{L}\tfrac{G'(Lz_1)}{\pi-G(Lz_1)}\partial_{z_1} (z_2-1)\partial_{z_2} \right.\\
    & \ \ \left.+ \left( \tfrac{G'(Lz_1)}{\pi-G(Lz_1)} \right)^2 \left( (z_2-1)\partial_{z_2} \right)^2 \right]+\tfrac{4(1-\omega^2)}{(\pi-G(Lz_1))^2} \partial_{z_2}^2.
    \end{split}\]
    One can further use the parametrization $(z_1, z_2)=(\gamma_\tau(y_1), y_2)$ and calculate $P_{\Gamma}(\omega)$ in $(y_1, y_2)$ by replacing $z_1, z_2, \partial_{z_1}, \partial_{z_2}$ with $\gamma_\tau(y_1), y_2, \gamma_\tau'(y_1)^{-1}\partial_{y_1}, \partial_{y_2}$ respectively. It remains to solve \eqref{eq:complex} numerically using Chebyshev differentiation matrices, see \cite{Tr:00}. In Figure \ref{fig:sub_osc}, we choose $L=15$, $\tau=0.5$, and 
    \[\begin{split} 
    G(x_1) = e^{-\frac{1}{10}x_1^2}, \
    \rho(y_1) = 0.5\times(2+\tanh{(20(y_1-0.9))}-\tanh{(20( y_1+0.9 ))} ). 
    \end{split}\]
\end{Remark}

\subsection{Relation to oceanography literature}

Equation \eqref{eq:IW} describes the generation of linear internal tides in the ocean, ignoring the Coriolis force (i.e., the rotation of the Earth). The driving force of internal tides is provided by the barotropic tidal flow over the topography. In \cite{Ba:82}, the effect of barotropic tidal flow is modeled by a buoyancy force and the forcing term takes the form (see \cite[(2.7)]{Ba:82} with $f=0$ there) 
\[ QN(x_2)^2 x_2 \partial_{x_1}^2\left( \tfrac{1}{\pi-G(x_1)} \right) \cos(\lambda t) \]
where $Q$ is a constant, $N(x_2)$ is the buoyancy frequency describing the density stratification. Baines looked for standing wave solutions $u(t,x)=U(x)\cos(\lambda t)$, where $U(x)$ solves the stationary equation with the ``radiation conditions'' for $|x_1|$ large. The radiation condition excludes incoming waves traveling towards the bottom topography, that is, the ``generation region''. The standing wave solution is exactly the leading profile $\mathrm{Re}(e^{i\lambda t} u^+)$ that we defined in Theorem \ref{thm:tide}. However, the standing wave does not satisfy the vanishing initial value conditions, nor did Baines prove the decomposition in Theorem \ref{thm:tide} for the initial value problem. In particular, the proof of Theorem \ref{thm:tide} implies that it takes infinite time for the full development of the leading outgoing profile.

Notice that Baines' model assumed that the barotropic flow is hydrostatic (hence it only depends on $x_2$). A model for baroclinic tidal flows, valid under both hydrostatic and non-hydrostatic assumptions, is derived in \cite[(17)--(19)]{GaGe:07}. We remark that it is also possible to model the forcing through boundary conditions, and we refer to \cite[Chapter 2]{VlStHu:05} for details and references.

Finally, we remark that it is possible to include the Coriolis force in our analysis. In that case, the governing linear equation is 
\[ \left( \partial_t^2\Delta + \partial_{x_1}^2+a^2 \partial_{x_2}^2 \right)u(t,x) = f(x)\cos{\lambda t}, \ u|_{t=0}=\partial_t u|_{t=0}=0, \ u|_{\partial\Omega}=0 \]
with $a>0$ and $a\neq 1$.
Upon assuming 
\[ \min\{1,a\}<\lambda<\max\{1,a\} \text{ and } \max_{x_1\in \R}|G'(x_1)|<\left( \tfrac{\lambda^2-a^2}{1-\lambda^2} \right)^{\frac12}, \]
our analysis still applies.

\subsection{Structure of the paper}
The manuscript is organized in the following way.
In \S \ref{sec:pre} we review necessary microlocal analysis tools, in particular the scattering calculus. We also review the related chess billiard dynamics. We give an alternative characterization of incoming/outgoing solutions to stationary internal wave equations, which is more convenient for our analysis compared to the (equivalent) definition in \cite{LiWaWu:24}. In \S \ref{sec:end} we prove end estimates, showing the Neumann data for solutions to the stationary internal wave equation \eqref{eq:P(w)u} near ends of the channel can be controlled by the solution in a compact region of the bulk of the channel. In \S \ref{sec:reduction} we introduce boundary reduced operators. In particular, we show the microlocal structure of the restricted single layer potential $\mathrm d\mathcal C_\omega$. In \S \ref{sec:global} we prove propagation estimates in compact regions for boundary reduced operators. Combining end estimates in \S \ref{sec:end}, we obtain global high frequency estimates. In \S \ref{sec:LAP} we show a limiting absorption principle for solutions to the stationary equation \eqref{eq:P(w)u} as $\Im \omega\to 0+$ and prove the main results Theorem \ref{thm:tide} and \ref{thm:spectral}.

\vspace{8pt}
\noindent
{\bf Acknowledgments.} We would like to thank Semyon Dyatlov and Maciej Zworski for helpful discussions on this project. Z.~Li acknowledges partial support from NSF grant DMS-2400090. J.~Wang is supported by the Simons Foundation through a postdoctoral position at the Institut des Hautes \'Etudes Scientifiques.
J.~Wunsch acknowledges partial support from NSF grants DMS-2054424 and DMS-2452331 and from Simons Foundation grant MPS-TSM-00007464.

\section{Preliminaries}\label{sec:pre}
We first recall some convenient microlocal machinery used in this paper. In particular, we recall some basic properties of the scattering pseudodifferential calculus, which readily gives us mapping properties between weighted Sobolev spaces and allows us to localize to frequencies near spatial infinity. We refer the reader to~\cite{Va:18} for a complete treatment of the scattering calculus. Then we highlight the key uniqueness result from~\cite{LiWaWu:24} needed to develop a limiting absorption principle. 
\subsection{Microlocal analysis}
We say that $a \in C^\infty(T^* \R)$ is a scattering symbol of order-($m,l$) for $m,l\in \R$, denoted $a \in S^{m, l}(T^*\R)$, if it satisfies the estimates
\[|\partial_{\theta'}^\alpha \partial_{\xi}^\beta a(\theta', \xi)| \le C_{\alpha, \beta} \langle \xi \rangle^{m - \beta} \langle \theta' \rangle^{l - \alpha}, \ \text{ for all } (\theta',\xi)\in T^*\R, \ \alpha,\beta\in \mathbb N_0.\]
Here $\langle \xi \rangle:=\sqrt{ 1+\xi^2 }$.
The residual class of the scattering symbols is denoted by
\[S^{-\infty, -\infty}(T^* \R) \coloneqq \bigcap_{m, l \in \R} S^{m, l}(T^* \R).\]
The class of order-$(m, l)$ scattering pseudodifferential operators on $\R$, denoted $\Psi^{m, l}(\R)$, are given by the quantization of $a \in S^{m, l}(T^*\R)$
\[\Op(a)u(\theta) \coloneqq \frac{1}{2\pi} \iint e^{i(\theta - \theta') \xi} a(\theta', \xi) u(\theta')\, \md\theta' \md\xi.\]
We note that this is the \textit{right} quantization as opposed to the more commonly seen left quantization. This is to simplify some symbol computations in \S\ref{sec:kernel}.

It is clear that the action of these operators on $\CIc(\R)$ is well-defined. They further extend to operators 
\[\Op(a): H^{s, \alpha}(\R) \to H^{s - m, \alpha - l}(\R).\]
Here, $H^{s, \alpha}(\R)$ denote the weighted Sobolev spaces defined by 
\begin{equation}\label{eq:weighted_sob}
    H^{s, \alpha}(\R) \coloneqq \{\langle x \rangle^\alpha u: u \in H^s(\R)\}, \qquad \|u\|_{H^{s, \alpha}(\R)} \coloneqq \|\langle x \rangle^\alpha u\|_{H^s(\R)}.
\end{equation}
Note that $H^{s, \alpha}(\R) \hookrightarrow H^{s', \alpha'}(\R)$ compactly provided that $s > s'$, $\alpha > \alpha'$. 
A subtler feature of the pseudodifferential calculus is that if $A \in \Psi^{m,l}(\R)$ is globally \emph{elliptic} in the sense that $A=\Op(a)$, with $a$ invertible in $S^{m,l}(T^*\R)$ modulo $S^{m-1,l-1}(T^*\R)$, then 
$$
Au\in L^2 (\R) \Longleftrightarrow u \in H^{m,l}(\R).
$$

Two important operations of scattering pseudodifferential operators are composition and change of variables. Let $A = \Op(a) \in \Psi^{m, l}(\R)$ and $B = \Op(b) \in \Psi^{m', l'}(\R)$. Then $A \circ B \in \Psi^{m + m', l + l'}(\R)$. Furthermore, $A \circ B = \Op(c)$ where $c$ is given by the asymptotic expansion
\begin{equation}\label{eq:composition}
    c(\theta', \xi) \sim \sum_{k} \frac{((i\partial_{\theta'})^k a(\theta', \xi)) (\partial_\xi^k b(\theta', \xi))}{k!}.
\end{equation}
Let $\varkappa: \R \to \R$ be a diffeomorphism such that $\varkappa'(\theta) = 1$ for $|\theta| > M$ for some $M > 0$. Then 
\[\varkappa^{*} A \varkappa^{-*} = \Op(\tilde a) \in \Psi^{m, \ell}(\R),\]
where $\varkappa^{-*} := (\varkappa^{-1})^*$, and $\tilde a$ has the asymptotic expansion
\begin{equation}\label{eq:cov}
    \tilde a(\theta', \xi) \sim \sum_{k = 0}^\infty (L_{k} a)(\varkappa(\theta'), \varkappa'(\theta')^{-1} \xi)
\end{equation}
where $L_{k}$ are order $2k$ differential operators of the form 
\[L_0 = \Id, \qquad L_k = \sum_{\ell = 0}^k c_\ell(\theta', \xi) \partial_{\xi}^{\ell + k} \quad \text{for $k \ge 1$,}\]
where $c_\ell(\theta', \xi) \in C^\infty(\R^2)$ are order-$\ell$ symbols in $\xi$ and compactly supported in $\xi$. In particular, note that for every $k \ge 1$, the $k$-th term in the expansion~\eqref{eq:cov} is an order $k$ symbol in $\xi$ and compactly supported in $\theta'$ assuming that $\varkappa'(\theta) = 1$ for $|\theta| > M$.

\subsubsection{Cutoff symbols}
Note that the Schwartz kernel of an operator $A \in \Psi^{s, 0}(\R)$ has singular support contained in the diagonal $\{\theta = \theta'\}$ and satisfies the off-diagonal estimate for all $\alpha, \beta\in \mathbb N_0$
\begin{equation}\label{eq:rapid_off_diag}
    |\partial_\theta^\alpha \partial_{\theta'}^\beta A(\theta, \theta')| \le C_{N, \alpha, \beta} \langle \theta - \theta'\rangle^{-N} \text{ for all } N\in \R \ \text{when} \ |\theta - \theta'| \ge 1.
\end{equation}
This rapid off-diagonal decay is an important feature to ensure that the action on weighted Sobolev spaces makes sense. The off-diagonal decay comes from the fact that the symbol is smooth and gains decay when differentiated in $\xi$. However, we also need to consider slightly more exotic symbols that are singular near the zero section of $T^*\R$. Let $z \in C^\infty(\R)$ be such that 
\[\partial_{\theta'} z(\theta') = 0 \quad \text{for} \quad z \ge M\]
for some $M > 0$. Consider the $\epsilon$-dependent family of operators $\Op(b)$ where
\begin{equation}\label{eq:cutoff_symbol}
    b(\theta', \xi) = H(\xi) e^{-\epsilon|\xi| z(\theta')}, \ \epsilon>0
\end{equation}
where $H$ denotes the Heaviside function. To analyze such a symbol, we split it into two parts using smooth cutoff functions
\begin{equation}\label{eq:freq_cutoff}
    h_+ \in C^\infty(\R; \R), \quad \supp h_+ \subset (0, \infty], \quad h_+(\xi) = 1 \quad \text{when} \quad \xi \ge c_\lambda.
\end{equation}
and $h_-(\xi) \coloneqq h_+(-\xi)$.
We denote the corresponding frequency projections by 
\begin{equation}\label{eq:proj}
    \widetilde \Pi_\pm = \Op(h_\pm) \in \Psi^{0, 0}(\R).
\end{equation}
The tilde here is to emphasize that the Fourier character of this projection is smooth.

Henceforth, we will often use the same letter to denote both the operator and its Schwartz kernel. We define a modified residual class $\Psi_{\mathrm{res}}(\R) $ of operators such that $R \in \Psi_{\mathrm{res}}(\R)$ if and only if $R \in C^\infty(\R \times \R)$ and for all $\alpha, \beta\in \mathbb N_0$
\begin{equation}\label{eq:kernel_decay}
        |\partial_\theta^\alpha \partial_{\theta'}^\beta R(\theta, \theta')| \le C_{\alpha, \beta} \langle \theta - \theta' \rangle^{-1} \text{ for all } \theta, \theta'\in \RR.
    \end{equation}
If $R = R_\epsilon$ is an $\epsilon$-dependent family, then we say that $R \in \Psi_{\mathrm{res}}(\R)$ uniformly if and only if the family satisfies~\eqref{eq:kernel_decay} uniformly. Note that for every $\chi \in \CIc(\R)$ and $N \in \R$, we have the mapping property
\[\chi R \chi : H^{-N} \to H^N.\]
The mapping property is clearly uniform if $R = R_\epsilon \in \Psi_{\mathrm{res}}(\R)$ uniformly. 

The precise rate of off-diagonal decay of the Schwartz kernel of operators in $\Psi_{\mathrm{res}}(\R)$ is actually not so important for us, but we need some decay so that the following holds. 

\begin{lemma}\label{lem:residual}
Suppose $R \in \Psi_{\mathrm{res}}(\R)$, then for any $A \in \Psi^{s, 0}(\R)$, 
\[AR,\, RA \in \Psi_{\mathrm{res}}(\R).\]
\end{lemma}
\begin{proof}
    Note that it suffices to check the estimate~\eqref{eq:kernel_decay} for $\alpha = \beta = 0$ since $\partial_\theta A, A \partial_\theta \in \Psi^{s + 1, 0}(\R)$. Let $A = \Op(a) \in \Psi^{s, 0}(\R)$ and let $\chi \in \CIc(\R; [0, 1])$ be such that $\chi = 1$ near $0$. The Schwartz kernel of $AR$ is given by 
    \[
        AR(\theta, \theta') = \frac{1}{2 \pi} \iint e^{i(\theta - \theta'')\xi} a(\theta'', \xi) R(\theta'', \theta') \, \md \xi \md \theta''.
    \]
    Observe that 
    \begin{align*}
        &\left|\iint \chi(\xi) e^{i(\theta - \theta'')\xi} a(\theta'', \xi) R(\theta'', \theta') \, \md \xi \md \theta'' \right| \\
        & \lesssim \int \chi(\theta - \theta'') |R(\theta'', \theta')| \, \md \theta''  + \int \frac{1 - \chi(\theta - \theta'')}{|\theta - \theta''|^N} |R(\theta'', \theta')| \, \md \theta'' \\
        & \lesssim \langle \theta - \theta' \rangle^{-1}, 
    \end{align*}
    where we integrated by parts in $\xi$ for the first inequality and the implicit constants depend only on symbol seminorms of $a$, seminorms of $\chi$, and the constants $C_{\alpha, \beta}$ in the estimates~\eqref{eq:kernel_decay} for $R$. Similarly, we also see that 
    \begin{align*}
        &\left|\iint (1 - \chi(\xi)) e^{i(\theta - \theta'')\xi} a(\theta'', \xi) R(\theta'', \theta') \, \md \xi \md \theta'' \right| \\
        & = \left|\iint \frac{(1 - \chi(\xi))}{\xi^N} e^{i(\theta - \theta'')\xi} D_{\theta''}^N (a(\theta'', \xi) R(\theta'', \theta')) \, \md \xi \md \theta'' \right|  \\
        & \lesssim \int \chi(\theta - \theta'') \langle \theta'' - \theta' \rangle^{-1} \, \md \theta''  + \int \frac{1 - \chi(\theta - \theta'')}{|\theta - \theta''|^N} \langle \theta'' - \theta' \rangle^{-1} \, \md \theta'' \\
        & \lesssim \langle \theta - \theta' \rangle^{-1}.  
    \end{align*}
    Estimates for $RA$ follow by a similar argument. 
\end{proof}

\begin{lemma}\label{lem:cutoff_decomp}
    Let $b$ be a cutoff symbol of the form~\eqref{eq:cutoff_symbol}. Then the Schwartz kernel $R$ of $\Op((H(\xi) - h_+(\xi))b)$ lies in $\Psi_{\mathrm{res}}(\R)$ uniformly. 
\end{lemma}
\begin{proof}
    We wish to estimate kernel of the form 
    \begin{equation*}
        \frac{1}{2 \pi} \iint e^{i(\theta - \theta' + i \epsilon z(\theta'))\xi} (H(\xi) - h_+(\xi))\, \md \xi.
    \end{equation*}
    The inverse Fourier transform of $H(\xi) - h_+(\xi)$ extends to an entire function with estimates 
    \[ \left|\partial_z^k \mathcal F^{-1} (H - h_+)(z)\right| \le C \langle x \rangle^{-1 - k} \quad \text{for $|y| \le 1$ where $z = x + iy$},\]
    from which the lemma follows. 
\end{proof}

\subsubsection{Sobolev spaces}\label{sssection:sobolev}
We fix some notations for various Sobolev spaces used in this paper. Recall the weighted Sobolev spaces $H^{s, \alpha}(\R)$ defined in~\eqref{eq:weighted_sob}. We can make the identification $\partial \Omega \simeq \R \sqcup \R$ using the $x_1$ coordinate. Then we can define the weighted Sobolev spaces $H^{s, \alpha}(\partial \Omega)$ simply as two copies of $H^{s, \alpha}(\R)$ on each component of the boundary. 

On the whole domain $\Omega$, we can consider subspaces of the extendable distributions $\overline{\mathcal D}'(\Omega)$ (see \cite[Appendix B]{Ho:85} for a review of distributions on manifolds with boundary). We have the usual subspaces $\bar H^s_{\loc}(\Omega)$ and $\bar H^s_{\comp}(\Omega)$ denoting the local and compactly supported extendable Sobolev spaces respectively. We can also globally define $\bar H^s(\Omega)$ using the Euclidean metric descended from $\R^2$. We will also consider Sobolev spaces weighted in the $x_1$ direction and define
\begin{equation*}
    \bar H^{s, \alpha}(\Omega) \coloneqq \{u \in \bar H^{s}_{\loc}(\Omega) : \langle x_1 \rangle^{\alpha} u \in \bar H^{s} \}, \qquad \|u\|_{\bar H^{s, \alpha}(\Omega)} \coloneqq \|\langle x_1 \rangle^{\alpha} u\|_{\bar H^s(\Omega)}.
\end{equation*}
Since we are interested in Dirichlet boundary conditions, we will also use the space $H^1_0(\Omega)$, denoting the space of distributions on $\Omega$ which extend by zero to elements of $H^1(\R^2)$. Equivalently, these are elements of $\bar H^1(\Omega)$ which have zero boundary trace. Finally, we can define the closed subspaces $\bar H^{s, \alpha}_0(\Omega) \coloneqq \bar H^{s, \alpha}(\Omega) \cap H^1_{0, \loc}(\Omega)$ for $s \ge 1$. Strictly speaking, we can define these spaces down to $s > 1/2$, but there is no need since $H^1_0(\Omega)$ is the natural energy space in view of the spectral theory.

\subsection{Dynamics and scattering}\label{sec:dynamics}
Recall that we are interested in studying the resolvent for the operator $P$ defined in~\eqref{eq:0pseudo}. Using the invertibility of the Dirichlet Laplacian, this amounts to studying the equation 
\begin{equation}\label{eq:P(w)u}
\begin{gathered}
    P(\omega)u_\omega = f, \qquad f \in \CIc(\Omega),  \\
    P(\omega) \coloneqq (P - \omega^2)\Delta = (1 - \omega^2) \partial_{x_2}^2 - \omega^2 \partial_{x_1}^2.
\end{gathered}
\end{equation}
By the spectral theorem, when $\Im \omega \neq 0$, there exists a unique $u_\omega \in H^1_0(\Omega)$ that solves~\eqref{eq:P(w)u}. For the limiting absorption principle, we are interested in the behavior of $u_\omega$ as $\omega \to \lambda \in [0, 1]$. We have the limiting problem
\begin{equation}\label{eq:SIW}
    P(\lambda) u = f, \quad u|_{\partial \Omega} = 0, \quad f \in \CIc(\Omega),
\end{equation}
which we call the \emph{stationary internal wave equation}. Observe that the limiting operator $P(\lambda)$ is simply a $(1 + 1)$-dimensional wave operator. Define
\begin{equation}\label{eq:ellpm}
    \ell^\pm(x, \lambda) \coloneqq \pm \frac{x_1}{\lambda} + \frac{x_2}{\sqrt{1 - \lambda^2}}
\end{equation}
The level sets of $\ell^\pm(x, \lambda)$ are precisely the characteristic sets of $P(\lambda)$, and these characteristic lines have slope $\pm c_\lambda$ with
\begin{equation}\label{eq:slope}
    c_\lambda \coloneqq \frac{\sqrt{1 - \lambda^2}}{\lambda}.
\end{equation}
The $\lambda$-subcritical condition~\eqref{eq:subcrit} can then be interpreted in terms of the characteristics. Define the upper and lower boundary of $\partial \Omega$ by 
\begin{equation*}
    \partial\Omega_\uparrow\coloneqq \{(x_1, 0) : x_1 \in \R\}, \quad \partial \Omega_{\downarrow} \coloneqq \{(x_1, G(x_1) - \pi) : x_1 \in \R\}.
\end{equation*}
If $\Omega$ is $\lambda$-subcritical, then each characteristic intersects each of $\partial \Omega_\uparrow$ and $\partial \Omega_\downarrow$ transversally at precisely one location. Therefore, there exist unique nontrivial involutions $\gamma_\lambda^\pm(x): \partial \Omega \to \partial \Omega$ so that 
\[\ell^\pm(\gamma^\pm_\lambda(x), \lambda) = \ell^\pm(x, \lambda).\]
Note that $\gamma_\lambda^\pm: \partial \Omega_\uparrow \to \partial \Omega_{\downarrow}$ and also $ \gamma_{\lambda}^{\pm}: \partial \Omega_\downarrow \to \partial \Omega_\uparrow$. Composing the two involutions then yields a map $b_\lambda: \partial \Omega_\bullet \to \partial \Omega_\bullet$, $\bullet=\uparrow, \downarrow$, given by 
\begin{equation*}
    b_\lambda = \gamma^-_\lambda \circ \gamma^+_\lambda.
\end{equation*}
We will often drop the subscript $\lambda$ from these maps since it will be clear from context. See Figure~\ref{fig:characteristics} for an diagram of $\Omega$ and the dynamics. 

\begin{figure}
    \centering
    \includegraphics{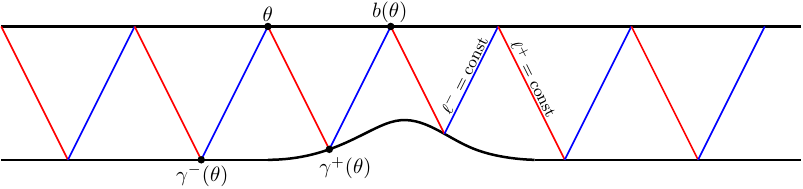}
    \caption{Example of a subcritical channel with the dynamics and characteristics labeled.}
    \label{fig:characteristics}
\end{figure}

Informally speaking, scattering data is propagated by the map $b_\lambda$, and the $\lambda$-subcritical condition essentially guarantees that scattering data on the far left side of the domain gets propagated to the far right side with no back scattering (and vice-versa). A precise description of the scattering problem is given in \cite{LiWaWu:24}, where we gave the exact relationship between the incoming and outgoing data on the ends of the domain. 

Fix some large $M > 0$ so that
\begin{equation}\label{eq:M}
    \supp G \subset \{x_1 \in \R: |x_1| < M\}, \quad \supp f \subset \{x \in \Omega: |x_1| < M\}.
\end{equation}
We will refer to the regions $\Omega \cap \{ \pm x_1 > M\}$ the right and left end of the domain respectively. Note that we simply have $P(\omega) u_\omega(x) = 0$ for $x$ in the left or right end, so we can solve the equation using Fourier sine series in the $x_2$ direction. The scattering data is located on the left and right ends of the domain, and it was shown in~\cite{LiWaWu:24} that incoming (and outoing) solutions are unique. A consequence of the uniqueness theorem is the following theorem, which shows that a condition roughly analogous to the Sommerfeld radiation condition in conventional scattering theory is sufficient to obtain uniqueness of solutions.
\begin{proposition}\label{prop:uniqueness}
    Fix cutoffs $\chi_\pm \in C^\infty(\R_{x_1})$ such that
    \[\supp \chi_\pm \subset \pm (M, \infty], \quad \chi_\pm(x_1) = 1 \text{ when } \pm x_1 \ge M + 1.\]
    Let $\widetilde \Pi_\pm$ denote Fourier projections defined in~\eqref{eq:proj}. Suppose $u \in H^2_{\loc}(\Omega) \cap H^1_{0, \mathrm{loc}}(\Omega)$ is a solution to $P(\lambda)u = 0$. Let $v \in L^2_{\loc}(\partial \Omega)$ denote the Neumann data. If
    \[\|\widetilde \Pi_+ \chi_+ v\|_{H^{0, \alpha}(\partial \Omega)} + \|\widetilde \Pi_- \chi_- v\|_{H^{0, \alpha}(\partial \Omega)} < \infty\]
    for some $\alpha > -1/2$, then $u = 0$. 
\end{proposition}
\begin{proof}
    Solving in Fourier series and taking the Neumann data, we see that 
    \begin{equation*}
        v(x_1) = \sum_{k \in \N} \left(a^-_k e^{-ic_\lambda k x_1} + a^+_k e^{ic_\lambda k x_1}\right), \quad x_1 \ge M
    \end{equation*}
    for some $(a^\pm_k) \in \ell^2$. 
Let us denote 
\[ \widetilde v(x_1)\coloneqq \sum_{k\in \mathbb N} (a_k^- e^{-ic_\lambda k x_1} + a_k^+ e^{i c_\lambda k x_1} ), \quad  x_1\in \R.   \]
Then we have 
\[ \widetilde \Pi_+\chi_+v= \widetilde \Pi_+\chi_+\widetilde v = [\widetilde \Pi_+,\chi_+]\widetilde v + \chi_+\widetilde \Pi_+\widetilde v. \]
Since $[\widetilde \Pi_+,\chi_+]\in \Psi^{-\infty,-\infty}$, the condition $\widetilde \Pi_+\chi_+ v\in H^{0,\alpha}$ is equivalent to $\chi_+\widetilde \Pi_+ \widetilde v\in H^{0,\alpha}$. We now compute 
\[\begin{split} 
\|\chi_+\widetilde \Pi_+\widetilde v\|_{H^{0,\alpha}}^2 
= & \int_{\R} \langle x_1 \rangle^{2\alpha}\chi_+(x_1)^2\left| \sum_{k\in \mathbb N} h_+(c_\lambda k) a_k^+ e^{ic_\lambda k x_1} \right|^2 \, \mathrm d x_1 \\
= & \sum_{k\in \mathbb N} | h_+(c_\lambda k)a_k^+|^2\int_{\R} \langle x_1 \rangle^{2\alpha}\chi_+(x_1)^2 \, \mathrm d x_1 \\
& + \sum_{j\neq k} h_+(c_\lambda j)h_{+}(c_\lambda k) a_j^+\overline{a_k^+}\int_{\R} e^{ic_\lambda (j-k)x_1}\langle x_1 \rangle^{2\alpha}\chi_+(x_1)^2 \, \mathrm d x_1.
\end{split}\]
For $\alpha>-\frac12$, $j\neq k$, we have 
\[\begin{split} 
& \int_{\R} \langle x_1 \rangle^{2\alpha}\chi_+(x_1)^2 \,\mathrm d x_1 = +\infty, \\ & \int_{\R} e^{ic_\lambda (j-k)x_1} \langle x_1 \rangle^{2\alpha}\chi_+(x_1)^2 \, \mathrm d x_1 = \mathcal O(|j-k|^{-N})  \text{ for all } N\in \R. 
\end{split}\]
The second estimate is obtained by integration by parts in $x_1$. 
Recall that $h_+ = 1$ on $[c_\lambda,+\infty]$ and
we conclude that $\widetilde \Pi_+\chi_+ v\in H^{0,\alpha}(\partial\Omega)$ implies that $a_k^+=0$ for all $k\in \mathbb N$.
\end{proof}

Following the lemma, we give an equivalent formulation for~Definition~1.2 of \cite{LiWaWu:24}, where incoming and outgoing solutions were defined by exact vanishing of half the Fourier coefficients in the ends.
\begin{definition}\label{def:io}
    We say that $u \in \bar H^{2, -N}(\Omega)$ is \emph{incoming} if
    \[\|\widetilde \Pi_+ \chi_+ v\|_{H^{0, \alpha}(\partial \Omega)} + \|\widetilde \Pi_- \chi_- v\|_{H^{0, \alpha}(\partial \Omega)} < \infty\]
    for every $\alpha > -1/2$, where $v$ is the Neumann data of $u$. Similarly, we say that $u$ is \emph{outgoing} if 
    \[\|\widetilde \Pi_- \chi_+ v\|_{H^{0, \alpha}(\partial \Omega)} + \|\widetilde \Pi_+ \chi_- v\|_{H^{0, \alpha}(\partial \Omega)} < \infty\]
    for every $\alpha > -1/2$. 
\end{definition}
If $u$ is the the solution to the stationary internal waves equation~\eqref{eq:SIW}, then this definition coincides with the definition in \cite{LiWaWu:24}. However, we will also use this terminology with other functions on $\Omega$. Furthermore, note that incoming (or outgoing) functions form a subspace.

\section{End estimates}\label{sec:end}
Now we begin our analysis of the limiting resolvent \eqref{limitingresolvent}. Fix an open interval $\mathcal J \subset [0, 1]$ such that $\Omega$ is subcritical with respect to every $\lambda \in \mathcal J$. Let us consider the spectral parameter $\omega = \lambda + i \epsilon$ in the region $\mathcal J + i (0, \infty)$. In particular, we are interested in the behavior of the spectral problem as $\epsilon \downarrow 0$. We first consider the equation~\eqref{eq:P(w)u} in the left and right ends of our domain. See~\eqref{eq:M} for a precise definition of the left and right ends.

\subsection{High frequency decay estimate}
By the spectral theorem, we have that $P(\omega): H^1_0(\Omega) \to H^{-1}(\Omega)$ is invertible. Furthermore, by elliptic regularity, we know that if $u_\omega$ is a solution to \eqref{eq:P(w)u}, then $u_\omega \in H^1_0(\Omega) \cap \overline C^\infty(\Omega)$ (see \cite[\S20.1]{Ho:85}). Therefore, it makes sense to take the Neumann data of $u_\omega$,
\[v_\omega = (\partial_{\hat n} u_\omega)|_{\partial \Omega},\]
where $\hat n$ denotes the \emph{outward} pointing unit normal vector. %\?[In the proof of Proposition 3.1, used outward pointing instead of inward]

The goal of this subsection is to show that the positive frequencies of $v_\omega$ in the right end and the negative frequencies in the left end decay rapidly at fixed $\ep>0$, and that they are uniformly (as $\epsilon\downarrow 0$) controlled by low frequency and locally supported information about $u_\omega$. 
\begin{proposition}\label{prop:high_decay}
    Let $u_\omega \in H^1_0(\Omega)$ be a solution to \eqref{eq:P(w)u} for $f \in \CIc(\Omega)$,
    %\jared{Reiterate hypotheses on $f$?} 
    and let $v_\omega$ denote the corresponding Neumann data on $\partial \Omega_\bullet \simeq \R$, $\bullet = \uparrow, \downarrow$. Let $M$ be as in~\eqref{eq:M} and choose cutoffs $\chi_\pm \in C^{\infty}(\R;\R)$ so that
    \[ \chi_-(x_1) = \chi_+(-x_1), \ \supp \chi_+ \subset (M + 1, \infty], \ \chi_+(x_1) = 1 \ \text{when} \ x_1 \ge M + 2.\]
    Then for any $s, \alpha, N \in \R$ there exist $\chi_0(x_1) \in C_c^\infty(\R)$ and $C>0$ such that 
    \[\|\widetilde \Pi_\pm \chi_\pm v_\omega\|_{H^{s, \alpha}(\R)} \le C\|\chi_0 u_\omega\|_{L^2(\Omega)} + C \norm{ \chi_\pm v_\omega}_{H^{-N,-N}(\R)},\]
    where $\widetilde \Pi_{\pm}$ are the approximate Fourier projectors defined in~\eqref{eq:proj}.
\end{proposition}
\begin{Remark}
    This estimate should be interpreted as a form of a microlocal source estimate, which were first presented in the context of scattering theory in \cite{Me:94}. We emphasize that this is an a priori estimate; we already know that $u \in H^1_0(\Omega)$ exists. Implicit in this existence is that $u$ has sufficient decay in the ends. Such necessary decay usually manifests as a threshold condition in microlocal source estimates. 
\end{Remark}
\begin{proof}
    We first focus on the estimate in the right end of the domain. Consequently, we make the translation $x_1 \mapsto x_1 + M + 1$ so that the support of $f$ and the topography $G$ lies in the region $x_1 < -1$. In these coordinates, we note that 
    \[\supp \chi_+ \subset (0, \infty], \quad \chi_+ (x_1) = 1 \quad \text{when} \quad x_1 > 1, \quad \chi_0 \equiv 1 \quad \text{near} \quad x_1 = 0.\]
    We wish to study the equation in the region $x_1 > -1$. For $\epsilon > 0$, $P(\omega): H^1_0(\Omega) \to H^{-1}(\Omega)$ is invertible, so we can solve in Fourier series and see that for $x_1 > -1$, any solution must take the form 
    \begin{equation}\label{eq:u_series}
        u_\omega(x_1, x_2) = \sum_{k \in \N} a_k \sin(k x_2) e^{-i c_\omega k (x_1 + 1)}, \quad x_1 > -1
    \end{equation}
    where 
    \begin{equation*}
        c_\omega^2 \coloneqq \frac{1 - \omega^2}{\omega^2} \implies c_\omega = \frac{\sqrt{1 - \lambda^2}}{\lambda} - \frac{i \epsilon}{\sqrt{1 - \lambda^2}} + \mathcal O(\epsilon^2).
    \end{equation*}
    Note that there exists $\delta_\lambda>0$ such that $\Re c_\omega\geq \delta_\lambda$ for all $\ep$ sufficiently small.

    The Neumann data on $\partial \Omega_\uparrow$ in the right end is given by 
    \[v_\omega(x_1) = \sum_{k \in \N} k a_k e^{-ic_\omega k} e^{-i c_\omega k x_1}, \quad x_1 > -1.\]
    For $s\in \mathbb N_0$ we define $\tilde v_\omega \coloneqq (-\frac{1}{c_\omega}D_{x_1})^s v_\omega$, then for $x_1>-1$
    \[ \tilde v_\omega(x_1) = \sum_{k \in \N} \tilde a_k e^{-ic_\omega k x_1}, \quad \tilde a_k \coloneqq k^{s+1} a_k e^{-ic_\omega k}. \]
    Note that for $\ep\geq 0$, $\abs{\tilde a_k} \leq  |k|^{s+1} \abs{a_k}$.  It suffices to show 
    \[\|\widetilde \Pi_+ \chi_+ \tilde v_\omega\|_{H^{0, \alpha}(\R)} \le C \|\chi_0 u_\omega\|_{L^2(\Omega)} + C \|\chi_+\widetilde v_{\omega}\|_{H^{-N,-N}(\mathbb R)}.\]
  Repeated integration by parts with respect to $(\xi-i\eta)^{-1}D_{x_1}$ shows that for all $K \in \NN$
$$
\abs{\F(\ang{x_1}^\alpha \chi_+)(\xi-i\eta)} \leq C_{K} \ang{\xi}^{-K},\ \xi \in \supp h_+,\ \eta \in (0,+\infty).
$$
Now observe that
    \[\mathcal F(\ang{x_1}^{\alpha} \chi_+ \tilde v_\omega)(\xi) = \sum_{k \in \N} \tilde a_k \mathcal F(\ang{x_1}^{\alpha + N} \chi_+)(\xi + c_\omega k).\]
    Since $\Re c_\omega \geq \delta_\lambda$ for sufficiently small $\epsilon$, the Plancherel identity implies that for every $K \geq 2$,
\begin{equation*}\begin{aligned}
\norm{\widetilde \Pi_+  \ang{x_1}^\alpha \chi_+ \tilde v_\omega}_{L^2(\R)} &\leq  \sum_{k \in \N}\big \| h_+(\xi) \tilde a_k \mathcal F(\ang{x_1}^{\alpha} \chi_+)(\xi + c_\omega k)\big \|_{L^2(\R)}\\
&\lesssim \sum_{k \in \N}\big \| h_+(\xi) \abs{\tilde a_k} \ang{\xi + k\Re c_\omega}^{-K}\big \|_{L^2(\R)}\\
&\leq \sum_{k \in \N}\big \| h_+(\xi) \abs{\tilde a_k} \ang{\xi}^{-K/2} \ang{k\delta_\lambda}^{-K/2}\big \|_{L^2(\R)}\\
&\lesssim \sum_{k \in \N}\abs{\tilde a_k} \ang{k\delta_\lambda}^{-K/2}.
\end{aligned}
\end{equation*}
Using Cauchy--Schwarz inequality and taking $K$ sufficiently large, we may achieve
\[\begin{split}
\norm{\widetilde \Pi_+  \ang{x_1}^\alpha \chi_+ \tilde v_\omega}_{L^2(\R)}^2 \lesssim \left(\sum \langle k \rangle^{2s+2-K}\right)  \left(\sum \ang{k}^{-2s-2}\abs{\tilde a_k}^2 \right) \leq \sum \abs{a_k}^2 e^{2 k \Im c_\omega}.
\end{split}\]
Now take $\chi_0$ to equal $1$ on $(-1,1)$ in our shifted coordinates; this yields (by the Plancherel identity for sine series)
$$
\norm{\chi_0(x_1) u_\omega}^2_{L^2(\Omega)} \gtrsim  \sum \abs{a_k}^2 \int_{-1}^1 \abs{e^{-ik c_\omega(x_1+1)}}^2\, \md x_1 = \sum |a_k|^2 \frac{ 1- e^{4k\Im c_{\omega}} }{2 k \abs{\Im c_\omega}}.
$$
Notice that $\frac{ 1- e^{4k\Im c_{\omega}} }{2 k \abs{\Im c_\omega}} = e^{2k\Im c_
\omega}\frac{ \sinh(2k|\Im c_\omega|) }{2k|\Im c_\omega|}\geq e^{2k\Im c_\omega}$. Hence we obtain
$$
\norm{\widetilde \Pi_+  \ang{x_1}^\alpha \chi_+ \tilde v_\omega}_{L^2(\R)} \lesssim \norm{\chi_0(x_1) u_\omega}_{L^2(\Omega)}.
$$

Finally, since $[\langle x_1 \rangle^{\alpha}, \widetilde \Pi_+ ] \in \Psi^{-\infty, - \infty}$,
$$
\norm{\widetilde \Pi_+  \chi_+ \tilde v_\omega}_{H^{0,\alpha}(\R)} \lesssim \norm{\chi_0 u_\omega}_{L^2(\Omega)}+ \|\chi_+ v_\omega\|_{H^{-N, -N}(\R)} .
$$
 This completes the proof for the Neumann data on $\partial \Omega_{\uparrow}$ in the right end. The other cases follow by similar arguments with minor sign changes. 
\end{proof}

\subsection{Low frequency decay estimate}
Next, we show that $u_\omega$ is controlled in $H^{s, \beta}$ by $\chi_0 u_\omega$ for all $s\in \RR$ and $\beta < -\frac{1}{2}$, uniformly in $\epsilon$. This will be our analogue of a microlocal sink estimate.

\begin{proposition}\label{prop:low_decay}
    Let $u_\omega \in H^1_0(\Omega)$ be a solution to equation~\eqref{eq:P(w)u}, and let $\chi_\pm$ and $\chi_0$ be the same cutoffs as in Proposition~\ref{prop:high_decay}. Then for any $s\in \RR$, $\beta < -\frac{1}{2}$, we have
    \begin{equation}\label{eq:lowdecay}
        \|\chi_\pm u_\omega\|_{H^{s, \beta}(\Omega)} \le C\|\chi_0 u_\omega\|_{H^s(\Omega)}.
    \end{equation}
\end{proposition}
\begin{proof}
    Again, we focus on the right end of the domain and employ the change of coordinates $x_1 \mapsto x_1 + M + 1$ so that we have a series solution of the form~\eqref{eq:u_series}. For an even positive integer $s=2\ell$, applying $(-\pa_{x_1}^2-\pa_{x_2}^2)^\ell$ to the series solution \eqref{eq:u_series} and using the Plancherel identity in $x_2$ then yields for $\beta< -\frac 12$
    $$
    \|\chi_+ u_\omega\|_{H^{2\ell, \beta}(\Omega)}^2 \lesssim \sum_{k \in \N} k^{4\ell} |a_k|^2 e^{2(\Im c_\omega) k} \int (\chi_+(x_1))^2 \langle x_1 \rangle^{2\beta} e^{(2\Im c_\omega) k x_1} \, \md x_1.
    $$
    More generally, interpolation and duality shows that for any $s\in \RR$,
    for $\beta < -\frac{1}{2}$,
    \begin{equation}\label{eq:x1deriv}
    \begin{aligned}
        \|\chi_+ u_\omega\|_{H^{s, \beta}(\Omega)}^2 &\lesssim \sum_{k \in \N} k^{2s} |a_k|^2 e^{2(\Im c_\omega) k} \int (\chi_+(x_1))^2 \langle x_1 \rangle^{2\beta} e^{(2\Im c_\omega) k x_1} \, \md x_1 \\
        &\le \sum_{k \in \N} k^{2s} |a_k|^2 e^{2(\Im c_\omega)k} \int_0^\infty \langle x_1 \rangle^{2\beta} \, \md x_1 \\
        &\le C \sum_{k \in \N} k^{2s} |a_k|^2 \int \chi_0(x_1)^2 e^{2(\Im c_\omega) k(x_1 + 1)}\, \md x_1 \\
        &\le C \||\partial_{x_1}|^s (\chi_0 u_\omega)\|^2_{L^2} \\
        &\le C \|\chi_0 u_\omega\|^2_{H^s}.
    \end{aligned}
    \end{equation}
    We note that the constant $C$ here may depend on $\beta$, but is uniform in $\omega \in \mathcal J + i(0, \infty)$. The proof for the estimate in the left end of the domain follows by the same argument with minor sign changes.
\end{proof}

\begin{Remark}
    We point out a slight subtlety in the second to last inequality in~\eqref{eq:x1deriv}. It is extremely important that we take $x_1$ derivatives only. This is because if we look in the $x_2$ direction, the $H^1_0([0, \pi]) \cap \bar H^s([0, \pi])$ norm is not given by the weighted sum of square of the Fourier sine coefficients when $s$ is large; the latter is better interpreted as the norm in the domain of the $s/2$'th power of the Dirichlet Laplacian, which is in general a \emph{subspace} of the corresponding Sobolev space. As an example, consider $\varphi \in H^1_0([0, \pi]) \cap \bar H^4 ([0, \pi])$. Then we can expand $f(x) = \sum_{k \in \N} a_k \sin(kx)$. However, $\|\varphi\|_{H^4}^2 \neq \sum_{k \in \N} k^8 |a_k|^2$. For instance, one can consider a smooth extendable function on $[0, 1]$ such that $\varphi(x) = x^2$ for $x < 1/2$. Then one can check that $\sum_{k \in \N} k^8 |a_k|^2$ is not summable since the odd extension of $f$ to a periodic function on $(-\pi, \pi)$ does not lie in $H^4$. However, the Fourier characterization for the Sobolev norm is equivalent up to $H^1_0([0, \pi]) \cap H^{s}([0, \pi])$ for $s < 5/2$ (this follows from~\cite[Chapter 4, Lemma 5.4]{Ta:10}). In general, we only have the upper bound 
    \[\|\varphi\|_{H^s([0, \pi])}^2 \le \sum_{k \in \N} k^{2s} |\hat \varphi(k)|^2\]
    for $\varphi \in H^1_0([0, \pi]) \cap \bar H^s([0, \pi])$.
\end{Remark}

\subsection{Derivatives in the spectral parameter}
By the spectral theorem, $u_\omega$ is holomorphic in $\omega = \lambda + i \epsilon$ when $\epsilon > 0$. To obtain higher regularity of the spectral measure, we also need to understand derivatives in the spectral parameter. Observe that for $j \ge 1$, we have $\partial_\omega^j (P(\omega) u_\omega) = 0$, so 
\begin{equation}\label{eq:Pdu}
\begin{aligned}
    P(\omega) (\partial_\omega^j u_\omega) &= -\sum_{l = 1}^{j} \binom{j}{l} (\partial_\omega^l P(\omega))(\partial_\omega^{j - l} u_\omega) \\
    &= 2j\omega \Delta (\partial_\omega^{j - 1} u_\omega) + j(j - 1)\Delta (\partial_\omega^{j - 2} u_\omega).
\end{aligned}
\end{equation}
Inductively, it is then easy to see that $\partial_\omega u_\omega \in H^1_0(\Omega)$ is the unique solution to~\eqref{eq:Pdu} with $j=1$, that is,
\[ P(\omega) \partial_\omega u_\omega = 2\omega \Delta u_\omega, \ \partial_\omega u_\omega|_{\partial\Omega}=0. \]
We can again write explicit series solutions in the left and right ends of the domain and obtain analogous high and low frequency decay estimates to Proposition~\ref{prop:high_decay} and Proposition~\ref{eq:lowdecay}. 

\begin{proposition}\label{prop:high_decay_deriv}
    Let $u_\omega \in H^1_0(\Omega)$ be a solution to \eqref{eq:P(w)u}, and let $v_\omega$ denote the corresponding Neumann data on $\partial \Omega_\bullet \simeq \R$, $\bullet = \uparrow, \downarrow$. Choose cutoffs $\chi_\pm$ and $\chi_0$ as in Proposition~\ref{prop:high_decay}. Then for any $s, \alpha, N \in \R$, 
    \[\|\widetilde \Pi_\pm \chi_\pm \partial_\omega v_\omega\|_{H^{s, \alpha}(\R)} \le C\big(\|u_\omega\|_{H^{2, -N}(\Omega)} + \|\partial_\omega u_\omega\|_{H^{2, -N}(\Omega)} + \norm{ \chi_\pm v_\omega}_{H^{-N,-N}(\R)}\big).\]
\end{proposition}
\begin{proof}
    We focus on the right end and use the change of coordinates $x_1 \mapsto x_1 + M + 1$. Recall that $u_\omega$ has the series solution~\eqref{eq:u_series}. Since the solution is analytic in $\omega$ we can differentiate \eqref{eq:u_series} to see that, for $x_1>1$, the series solution for $\partial_\omega u_\omega$ is given by  
    \begin{equation}\begin{split}\label{eq:deriv_series}
        \partial_\omega u_\omega(x_1, x_2) 
        = &  \sum_{k \in \N} -\frac{i k a_k}{\omega^2 \sqrt{1 - \omega^2}}(x_1 + 1) \sin(kx_2) e^{-ic_\omega k(x_1 + 1)} \\
        & + \sum_{k \in \N} b_k \sin(kx_2) e^{-ic_\omega k(x_1 + 1)}
    \end{split}\end{equation}
    with $b_k \coloneqq \partial_\omega a_k$. The Neumann data on $\partial\Omega_{\uparrow}$ for $x_1>1$ is then given by 
    \[\partial_\omega v_\omega(x_1) = \sum_{k \in \N} -\frac{i k^2 a_k}{\omega^2 \sqrt{1 - \omega^2}}(x_1 + 1) e^{-ic_\omega k(x_1 + 1)} + \sum_{k \in \N} k b_k e^{-ic_\omega k(x_1 + 1)}.\]
    Taking derivatives, we consider, with $s\in \mathbb N_0$ and $x_1>1$,
    \[\tilde v(x_1) \coloneqq \Big( - \frac{1}{c_\omega} D_{x_1} \Big)^s \partial_\omega v_\omega(x_1) = \sum_{k \in \N} ( \tilde a_k(x_1 + 1) + \tilde b_k) e^{-i c_\omega k x_1} \]
    where 
    \[\tilde a_k \coloneqq -\frac{i}{\omega^2 \sqrt{1 - \omega^2}} k^{s + 2} a_k e^{-i c_\omega k}, \quad \tilde b_k \coloneqq k^{s + 1} b_k e^{-ic_\omega k}.\]
    Following the proof of Proposition~\ref{prop:high_decay}, we find the uniform pointwise bound %as~\eqref{eq:high_dec_pointwise}:
    \[|\mathcal F^{-1} \left(h_+(\bullet) \mathcal F(|x_1|^{\alpha + N} \chi_+ \tilde v)\right)|^2 \lesssim \sum_{k \in \N} \big(|\tilde a_k|^2 + |\tilde b_k|^2\big) \langle k \rangle^{-N + 2}\]
    for any $N \in \R$. 
    Indeed, integrating by parts with respect to $((\xi+c_\omega k)D_{x_1})^{N_1}$ gives 
    \[\begin{split} 
    |\mathcal F(|x_1|^{\alpha+N} \chi_+ \tilde v)(\xi)| 
    \leq & \sum_{k\in \mathbb N} \frac{1}{|\xi+c_\omega k|^{N_1}}\int \left| D_{x_1}^{N_1}\left( \chi_+ |x_1|^{\alpha + N}(\tilde a_k(x_1+1) + \tilde b_k) \right) \right| \, \md x_1 \\
    \lesssim & \sum_{k\in \mathbb N} \langle \xi+k \Re c_\omega \rangle^{-N_1}(|\tilde a_k| + |\tilde b_k|)
    \end{split}\]
    for $N_1\in \mathbb N$ large enough and $\xi\in \supp h_+$. Therefore 
    \[\begin{split} 
    & |\mathcal F^{-1} \left(h_+(\bullet) \mathcal F(|x_1|^{\alpha + N} \chi_+ \tilde v)\right)|^2 \\
    & \quad \lesssim \|h_+\mathcal F(|x_1|^{\alpha+N}\chi_+ \tilde v)\|_{L^1(\mathbb R)}^2 \\
    & \quad \leq \sum_{k\in \mathbb N} \langle k \rangle^{-N+2}(|\tilde a_k|^2 + |\tilde b_k|^2) \sum_{k\in \mathbb N} \langle k \rangle^{N-2}\left( \int \frac{h_+(\xi) \,\md \xi}{\langle \xi + k \Re c_\omega \rangle^{N_1}} \right)^2\\
    & \quad \lesssim \sum_{k\in \mathbb N} \langle k \rangle^{-N+2}(|\tilde a_k|^2 + |\tilde b_k|^2).
    \end{split}\]

    Now taking $N > 2s + 8$, it follows that 
    \[\begin{split} 
        & \|\langle x_1 \rangle^{-N}\widetilde \Pi_+\langle x_1 \rangle^{\alpha+N}\chi_+\tilde v\|^2_{L^2} \\
        & \quad \leq \sum_{k\in \mathbb N}\langle k \rangle^{2s+6-N}( |a_k|^2 + |b_k|^2 ) e^{k \Im c_\omega} \leq \sum_{k\in \mathbb N} \langle k\rangle^{-2} (|a_k|^2 + |b_k|^2) e^{2 k\Im c_\omega}. 
    \end{split}\]
    Notice that for all $x_1\in (-1,1)$, $k\in \mathbb N$,
    \[ |a_k|^2 + \left| -\frac{i k a_k}{\omega\sqrt{1-\omega^2}}(x_1+1) + b_k \right|^2 \gtrsim \langle k \rangle^{-2}(|a_k|^2+|b_k|^2), \]
    hence we estimate 
    \[ \|\chi_0 u_\omega\|^2_{L^2}+\| \chi_0\partial_\omega u_\omega \|_{L^2}^2 \gtrsim \sum_{k\in \mathbb N} \langle k \rangle^{-2}(|a_k|^2 + |b_k|^2)\frac{1-e^{4 k \Im c_\omega}}{k |\Im c_\omega|}.  \]
    Now we conclude that
    \begin{equation*}
        \|\langle x_1 \rangle^{-N} \widetilde \Pi_+ \langle x_1 \rangle^{\alpha + N} \chi_+ \tilde v\|_{L^2} \lesssim \|\chi_0 u_\omega\|_{L^2} + \|\chi_0 \partial_\omega u_\omega\|_{L^2}. 
    \end{equation*}
    Unwinding the definitions, we conclude that 
    \[\|\widetilde \Pi_+ \chi_+ \partial_\omega v_\omega \| \lesssim \|u_\omega\|_{H^{2, -N}(\Omega)} + \|\partial_\omega u\|_{H^{2, -N}(\Omega)} + \norm{ \chi_+ v_\omega}_{H^{-N,-N}(\R)}.\]
    Again, the left end estimates follow with minor sign changes in the argument. 
\end{proof}

We also have the low frequency decay estimates for the differentiated solution. In this case, the threshold condition changes.
\begin{proposition}\label{prop:low_decay_deriv}
    Let $u_\omega \in H^1_0(\Omega)$ be a solution to equation~\eqref{eq:P(w)u}, and let $\chi_\pm$ and $\chi_0$ be the same cutoffs as in Proposition~\ref{prop:high_decay}. Then for any $\beta < -\frac{3}{2}$, we have
    \begin{equation}\label{eq:low_decay_deriv}
        \|\chi_\pm \partial_\omega u_\omega\|_{H^{s, \beta}(\Omega)} \le C\|\chi_0 \partial_\omega u_\omega\|_{H^{s + 1}(\Omega)} + \|\chi_0 u_\omega\|_{H^s(\Omega)}.
    \end{equation}
\end{proposition}
\begin{proof}
    We focus on the right end and change coordinates by $x_1 \mapsto x_1 + M + 1$ so that we have the series solution given by~\eqref{eq:deriv_series}. The desired inequality follows from a similar argument to Proposition~\ref{prop:low_decay} with a different threshold.
   We see that for $\beta < -\frac{3}{2}$, we have 
    \begin{equation*}
    \begin{aligned}
        \|\chi_+ u\|_{H^{s, \beta}(\Omega)}^2 &\lesssim \sum_{k \in \N} \big(k^{2s + 2} |a_k|^2 + k^{2s} |b_k|^2 \big) e^{2 \Im c_\omega k} \\
        &\lesssim \||\partial_{x_1}|^{s + 1} \chi_0 u_\omega\|_{L^2 (\Omega)}^2 + \||\partial_{x_1}|^s \chi_0 \partial_\omega u_\omega\|_{L^2(\Omega)}^2 \\
        &\lesssim \|\chi_0 u\|_{H^{s + 1}(\Omega)}^2 + \|\chi_0 \partial_\omega u_\omega \|_{H^s(\Omega)}^2.
    \end{aligned}
    \end{equation*}
    Left end estimates follow with minor sign changes in the argument. 
\end{proof}

\section{Boundary reduction}\label{sec:reduction}
With the end estimates in place, we now proceed to analyze the ``middle'' region that contains both the topography and the forcing. As in the previous section, we consider the equation~\eqref{eq:P(w)u} with the spectral parameter $\omega = \lambda + i \epsilon \in \mathcal J +i (0, \infty)$ where $\mathcal J \subset(0, 1)$ is a small interval such that $\Omega$ is $\lambda$-subcritical for every $\lambda \in \mathcal J$. 

Our approach for the middle section is based on black box methods in scattering theory (see \cite[\S 4]{DyZw:19} for an account of these methods in the context of Euclidean scattering). This then allows us to perform a reduction to the boundary similar to the strategy in \cite{DyWaZw:21, Li:24}. First, we choose a ``black box" cutoff that is much bigger than both the forcing and the topography. More precisely, choose $M>0$ such that
\begin{equation}\label{eq:bb1}
\begin{gathered}
    M > \frac{10(\pi - \min G)}{c_\lambda} \quad \text{for all} \quad \lambda \in I, \\
    \supp G \subset (-M, M), \quad \supp f \subset \{x \in \Omega: |x_1| < M\}.
\end{gathered}
\end{equation}
There exists an $N > 0$ so that 
\[b_\lambda^{\pm N}\big(\partial \Omega \cap \{|x_1| \le 4M\}\big) \cap \{|x_1| \le 4M\} = \emptyset \quad \text{for all} \quad \lambda \in \mathcal J.\]
In other words, under repeated applications of $b$, the top piece of boundary in $\{|x_1| \le 4M\}$ will keep moving right and the bottom piece will keep moving left, and at some point, both pieces will be disjoint to $\{|x_1| \le 4M\}$. Now choose $L > 0$ sufficiently large so that 
\begin{equation}\label{eq:bb2}
    b^{\pm N}\big(\partial \Omega \cap \{|x_1| \le 4M\}\big) \subset \{|x_1| \le L\}.
\end{equation}
Finally, we choose our black box cutoff $\chi_\mb$ so that 
\begin{equation}\label{eq:bb3}
    \chi_\mb \in \CIc(\R; \R), \quad \chi_\mb(x_1) = 1 \quad \text{for all} \quad |x_1| \le 2L.
\end{equation}
It is important that the black box cutoff is much bigger than the support of the topography and the forcing. This is to ensure that we can run propagation arguments near the support of the topography and the forcing without picking up effects coming from the edge of the black box. 

\begin{figure}[b]
    \centering
    \includegraphics{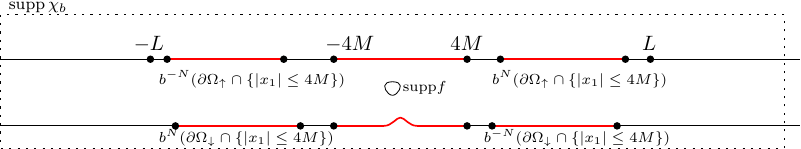}
    \caption{The scales $M,L$ shown relative to the supports of topography and inhomogeneity.}
    \label{fig:t}
\end{figure}

\subsection{Method of layer potentials}
  Recall from Definition~\ref{eq:ellpm} that we have defined coordinates
\begin{equation*}
    \ell^\pm(x, \omega) \coloneqq \pm \frac{x_1}{\omega} + \frac{x_2}{\sqrt{1 - \omega^2}};
\end{equation*}
in this section we extend the definition to allow for $\omega \in \CC$.  We also employ dual differential operators defined by
    \begin{equation}\label{Lpm} L^\pm_\omega \coloneqq \frac{1}{2} (\pm \omega \partial_{x_1} + \sqrt{1 - \omega^2} \partial_{x_2}).\end{equation}
    Thus,
    \[L^\pm_\omega \ell^\pm(x, \omega) = 1, \qquad L_\omega^\mp \ell^\pm(x, \omega) = 0.\]

Let $u_\omega \in H^1_0(\Omega)$ be a solution to~\eqref{eq:P(w)u}. We truncate $u_\omega$ by the black box cutoff $\chi_\mb$ and extend it to a function on $\R^2$ by putting $u_\omega=0$ in $\R^2\setminus \overline{\Omega}$. Then applying $P(\omega)$ and recalling that $u_\omega$ satisfies a Dirichlet boundary condition, we find that 
\begin{equation}\label{eq:P_on_cutoff}
\begin{aligned}
    P(\omega)( \indic_\Omega \chi_\mb u_\omega) &= \indic_\Omega P(\omega) \chi_\mb u_\omega - \mathcal I \mathcal N_\omega \chi_\mb u_\omega \\
    &= \indic_\Omega [P(\omega), \chi_\mb] u_\omega + f - \mathcal I \chi_\mb v_\omega,
\end{aligned}
\end{equation}
where $v_\omega := \mathcal N_\omega u_\omega$ is given by 
\[\mathcal N_\omega u_\omega = - 2 \omega \sqrt{1 - \omega^2} \mathbf j^* (L_\omega^+ u_\omega \md \ell^+(\bullet, \omega))\]
where $\mathbf j: \partial \Omega \to \overline \Omega$ is the embedding map, and $\mathcal I:  \mathcal E'(\partial \Omega; T^* \partial \Omega) \to \mathcal E(\R^2)$ is the unique operator such that 
\[\int_{\R^2} \mathcal I v(x) \varphi(x) \, \md x = \int_{\partial \Omega} \varphi v.\]
In other words, $\mathcal N_\omega$ is a scaled Neumann data operator, and $\mathcal I$ tensors a distribution on $\partial \Omega$ with the $\delta$-function on $\partial \Omega$ to yield a distribution on $\R^2$ supported on $\partial \Omega$. We note that in the left and right ends of $\Omega$, $\mathcal N_\omega$ is simply 
\[\mathcal N_\omega u_\omega = -(1 - \omega^2) \partial_{x_2} u_\omega(x_1, \bullet) \, \md x_1, \quad |x_1| > M, \ \bullet = 0, -\pi,\]
which is indeed a constant multiple of the Neumann data. So all estimates in \S\ref{sec:end} that apply to the Neumann data clearly apply to $\mathcal N_\omega u_\omega$. 

Now we use the fact that $P(\omega)$ on $\R^2$ is a constant coefficient elliptic operator.
The fundamental solution for $P(\omega)$ is given explicitly by 
\begin{equation}\label{eq:FS}
    E_\omega(x) = \kappa_\omega \log\big[ \ell^+(x, \omega) \ell^-(x, \omega)], \quad \kappa_\omega \coloneqq \frac{i \sgn \Im \omega}{4 \pi \omega \sqrt{1 - \omega^2}}
\end{equation}
where we take the branch of log on $\C \setminus (-\infty, 0]$ taking real values on $(0, \infty)$. See \cite[\S4.3]{DyWaZw:21} for details. 

Since $\indic_\Omega \chi_\mb u_\omega$ is compactly supported, convolving~\eqref{eq:P_on_cutoff} with the fundamental solution $E_\omega$ yields
\begin{equation}\label{eq:single_layer}
    \indic_\Omega \chi_\mb u_\omega = E_\omega * (\indic_\Omega [P(\omega), \chi_\mb] u_\omega) + E_\omega * f - S_\omega \chi_\mb v_\omega
\end{equation}
where $S_\omega \coloneqq (E_\omega *) \mathcal I$ is the \emph{single layer potential} that also appears in~\cite{DyWaZw:21, Li:24}. We need some mapping properties for $S_\omega$. 
\begin{lemma}\label{lem:S_mapping}
    Let $\mathcal J \subset (0, 1)$ be an open interval such that $\Omega$ is $\lambda$-subcritical for every $\lambda \in \mathcal J$. Then we have the mapping properties 
    \[S_\omega: H^s_{\comp} (\partial \Omega; T^* \partial \Omega) \to \bar H^{s + 1}_{\loc}(\Omega), \qquad \partial_\omega S_\omega: H^{s}_{\comp}(\partial \Omega; T^* \partial \Omega) \to \bar H_{\loc}^{s}(\Omega)\]
    uniformly in $\omega \in \mathcal J + i(0, \infty)$. 
\end{lemma}
\begin{Remark}
    With more work, one can show that $\partial_\omega^k S_\omega: H^{s + k}_{\comp}(\partial \Omega; T^* \partial \Omega) \to \bar H^{s + 1}_{\loc}(\Omega)$, but $C^1$ regularity of the spectral measure is more than enough for the evolution problem. 
\end{Remark}
\begin{proof}
    1. Let $v \in H^s_{\mathrm{comp}}(\partial \Omega; T^* \partial \Omega)$. 
    We check the mapping properties of the operators $L^\pm_\omega S_\omega v$, which is given by 
    \begin{equation}\label{eq:LS}
        L^\pm_\omega S_\omega v(x) = \kappa_\omega \int_{\partial \Omega} \frac{v(y)}{\ell^\pm(x - y, \omega)}.
    \end{equation}
    Let us focus on the $L^+_\omega$ case with $v$ supported on $\partial \Omega_{\downarrow}$. Parameterizing $v(y)$ using the $\theta' = y_1$ coordinate, we can write $v = v(\theta') \md\theta'$ for some $v \in H^s_{\comp}(\R)$. For the domain $\Omega$, we introduce smooth coordinates $(\theta, \tau)$ so that
    \[x_1 = \theta - \frac{\tau(\pi - G(\theta))}{c_\lambda}, \quad x_2 = - (1 - \tau)(\pi - G(\theta)).\]
    Note that $\overline \Omega = \R_\theta \times [0, 1]_\tau$ with $\partial \Omega_{\downarrow} = \{\tau = 0\}$.

    \noindent 
    2. In these coordinates, we can compute
    \begin{align*}
        &L^+_\omega S_\omega v(\theta, \tau) \\
        &= \kappa_\omega \int_\R \Big(\frac{\theta}{\omega} - \frac{\tau(\pi - G(\theta))}{c_\lambda \omega} - \frac{\theta'}{\omega} - \frac{(1 - \tau)(\pi - G(\theta))}{\sqrt{1 - \omega^2}} + \frac{\pi - G(\theta')}{\sqrt{1 - \omega^2}} \Big)^{-1} v(\theta')\, \md \theta' \\
        &= \kappa_\omega \int_\R \Big( \frac{\theta - \theta'}{\omega} + \frac{G(\theta) - G(\theta')}{\sqrt{1 - \omega^2}} - i \epsilon \tau z_\omega(\theta) \Big)^{-1} v(\theta')\, \md \theta'  
    \end{align*}
    where 
    \[z_\omega(\theta) = \frac{1}{i\epsilon} \Big(\frac{1}{\sqrt{1 - \omega^2}} - \frac{1}{c_\lambda \omega} \Big) (-\pi + G(\theta)) = \Big(\frac{1}{\lambda(1 - \lambda^2)^{3/2}} + \mathcal O(\epsilon)\Big) (-\pi + G(\theta)).\]
    Most importantly $z_\omega$ is smooth and bounded away from $0$ uniformly in $\omega$. 

    There exists $F \in C^\infty(\R^2)$ such that \begin{equation}\label{Fdef}G(\theta) - G(\theta') = (\theta - \theta') F(\theta, \theta').\end{equation} So we can write
    \[\begin{gathered}
    L^+_\omega S_\omega v(\theta, \tau) = \kappa_\omega \int_\R \big( \psi_\omega(\theta, \theta') (\theta - \theta') - i \epsilon \tau z_\omega(\theta) \big)^{-1} v(\theta') \, \md \theta', \\ \text{with } \psi_\omega(\theta, \theta') = \frac{1}{\omega} + \frac{F(\theta, \theta')}{\sqrt{1 - \omega^2}}.
    \end{gathered}\]
    By the subcriticality of $\lambda$, we see that for $\ep$ sufficiently small, 
    \[\Re \psi_\omega(\theta, \theta') \ge c > 0, \quad \Im \psi_\omega (\theta, \theta') = \mathcal O(\epsilon) \]
    uniformly in $\omega$. Therefore, $L^+_\omega S_\omega$ is a family of $0$-th order pseudodifferential operators in $\theta$ that depend smoothly on the parameters $\epsilon$ and $\tau$ uniformly in $\omega$; it is, consequently, bounded $H^s_{\comp} (\partial \Omega; T^* \partial \Omega) \to \bar H^{s}_{\loc}(\Omega)$. A similar conclusion follows from considering $L^-_\omega S_\omega$ and $v$ supported on $\partial \Omega_\uparrow$, from which the desired mapping property follows since $\pa_{x_j}$ are linear combinations of $L_\omega^\pm$ with uniformly bounded coefficients.

    \noindent 
    3. Next we examine the derivative in the spectral parameter. Observe that for $\Im \omega >0$,
    \[\partial_\omega \ell^\pm(x, \omega) = \frac{2\omega^2 - 1}{2 \omega(1 - \omega^2)} \ell^\pm(x, \omega) + \frac{1}{2 \omega(1 - \omega^2)} \ell^\mp(x, \omega).\]
    Then we can write
    \begin{equation}\label{eq:dE}
        \partial_\omega E_\omega = \frac{\partial_\omega \kappa_\omega}{\kappa_\omega} E_\omega + \kappa_\omega a_0 + \kappa_\omega a_+ \frac{\ell^+(x, \omega)}{\ell^-(x, \omega)} + \kappa_\omega a_- \frac{\ell^-(x, \omega)}{\ell^+(x, \omega)},
    \end{equation}
    where $a_0, a_\pm$ are uniformly bounded in $\omega \in \mathcal J + i(0, \infty)$. Therefore, 
    \[\partial_\omega S_\omega v(x) = \frac{\partial_\omega \kappa_\omega}{\kappa_\omega} S_\omega v(x) + \kappa_\omega a_0 \int_{\partial \Omega} v(y) + \kappa_\omega \sum_{\pm} a_\pm \int_{\partial \Omega} \frac{\ell^\pm(x - y, \omega) v(y)}{\ell^\mp(x - y, \omega)}.\]
    Observe that the last term takes the same form as~\eqref{eq:LS} modified by a smooth term in the numerator. Therefore, it has the same mapping property as $L_\omega^\pm S_\omega$, which gives $\partial_\omega S_\omega: H^{s}_{\comp}(\partial \Omega; T^* \partial \Omega) \to \bar H_{\loc}^{s}(\Omega)$. 
\end{proof}

It then follows that we have a well-defined map $\mathcal C_\omega : \CIc (\partial \Omega; T^* \partial \Omega) \to C^\infty (\partial \Omega)$ by restricting to the boundary:
\begin{equation*}
    \mathcal C_\omega v \coloneqq (S_\omega v)|_{\partial \Omega}.
\end{equation*}
It will be technically convenient to compose $\mathcal C_\omega$ with the differential to form $\md \mathcal C_\omega$. We remark that Lemma~\ref{lem:S_mapping} implies the mapping property $\md \mathcal C_\omega: H^s_{\comp}(\partial \Omega ; T^* \partial \Omega) \to H^{s - 1}_{\loc} (\partial \Omega; T^* \partial \Omega)$. However, we will see that this mapping property can be improved to $H^s_{\comp}(\partial \Omega ; T^* \partial \Omega) \to H^{s}_{\loc} (\partial \Omega; T^* \partial \Omega)$ uniformly in $\omega$. 

Now restricting~\eqref{eq:single_layer} to the boundary and taking the differential, we then find that
\begin{equation}\label{eq:bdr}
    r_\omega + g_\omega = \md \mathcal C_\omega(\chi_\mb v_\omega) 
\end{equation}
where
\[r_\omega \coloneqq \md\big((E_\omega * (\indic_\Omega [P(\omega), \chi_b] u_\omega))\big|_{\partial \Omega}\big), \quad g_\omega \coloneqq \md\big((E_\omega * f)\big|_{\partial \Omega}\big).\]

\begin{lemma}\label{lem:rw_est}
    Let $\chi \in \CIc(\partial \Omega)$ be such that $\supp \chi \subset \{|\theta| \le L\}$, and let $\widetilde \chi_\mb \in \CIc(\R; \R)$ be such that $\widetilde \chi_\mb = 1$ on $\supp \chi_\mb$. Then for all $s \in \R$, 
    \[\|\chi r_\omega \|_{H^s(\partial \Omega)} \le C\|\widetilde \chi_\mb u_\omega\|_{L^2(\Omega)}, \quad \|\chi \partial_\omega r_\omega\|_{H^s(\partial \Omega)} \le C\big(\|\widetilde \chi_\mb u_\omega\|_{L^2(\Omega)} +\|\widetilde \chi_\mb \partial_\omega u_\omega\|_{L^2(\Omega)}\big)\]
    uniformly in $\omega \in \mathcal J + i (0, \infty)$. 
\end{lemma}
\begin{figure}
    \centering
    \includegraphics[width=.9\textwidth]{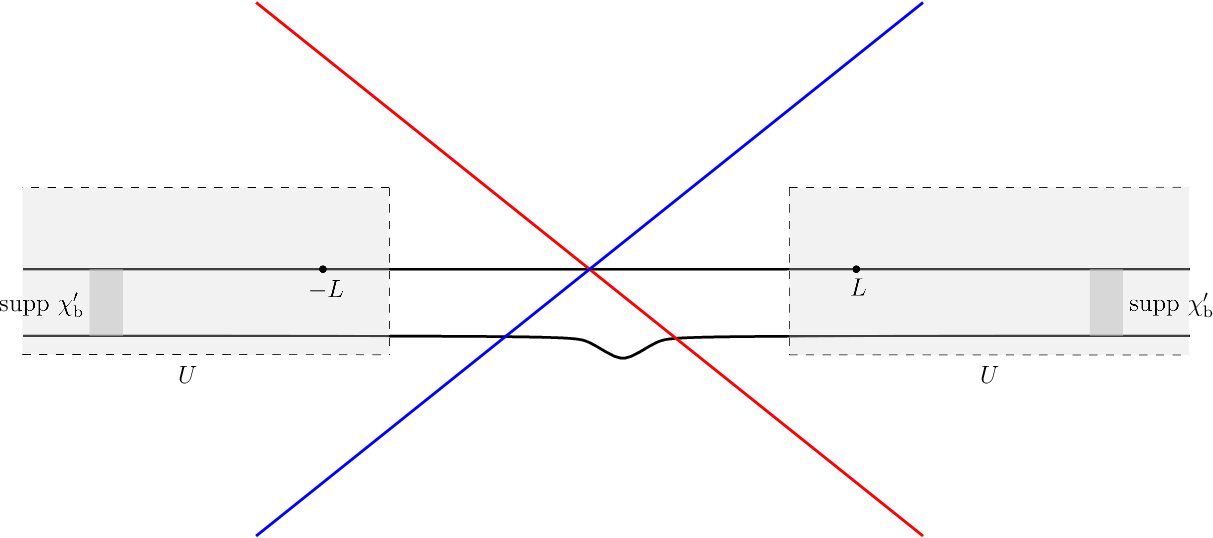}
    \caption{Diagram of the fundamental solution $E_\omega$ convolved with something supported on $\supp \chi'_\mb$. The red and blue lines are the singularities of $E_\omega$ given by $\{\ell^\pm(x, \lambda) = 0\}$ respectively, and the light gray regions are the region $U$ in the fundamental solution. The two dark gray patches near the ends denote $\supp \chi_\mb'$. It is clear from this picture that $\supp \chi_\mb'$ is contained in $U$ whenever the fundamental solution is centered in the region $\Omega_L$.}
    \label{fig:E_singularity}
\end{figure}
\begin{proof}
    The lemma follows by looking at the singularities of $E_\omega$ and $\partial_\omega E_\omega$. In particular, note that $E_\omega$ and $\partial_\omega E_\omega$ are uniformly smooth away from the sets $\{\ell^\pm(x, \lambda) = 0\}$. Consider the set 
    \[U = \{x \in \R^2: |x_1| > M,\, |x_2| < \pi - \min_{x_1 \in \R} G(x_1)\}.\]
    Recall that we chose $M$ to satisfy~\eqref{eq:bb1}, so $U$ is disjoint from $\{\ell^\pm(x, \lambda) = 0\}$. (See Figure~\ref{fig:E_singularity} for a diagram of this region with respect to singularities of $E_\omega$.) It is easy to see that $E_{\omega}|_{\overline U}, \partial_\omega E_{\omega}|_{\overline U} \in C^\infty(\overline U)$ uniformly in $\omega$. Furthermore, 
    \[\supp (\indic_\Omega [P_\omega, \chi_\mb] u_\omega) \subseteq \supp \chi'_\mb.\]
    Denote region $\Omega_L\coloneqq \{x \in \Omega: |x_1| \le L\}$ where $L$ is as in \eqref{eq:bb2}. Suppose $x\in \Omega_L$, $y\in \mathrm{supp} \ \chi'_\mb$, then $|x_1|<L$, $|y_1|>2L$, and $-\pi+\min G< x_2, y_2<0$, hence
    \[ |x_1-y_1|>L, \ |x_2-y_2|<\pi - \min G. \]
    In particular, this implies that $x-y\in U$. Hence we have 
    \[\left(E_\omega * (\indic_\Omega [P(\omega), \chi_\mb] u_\omega) \right)|_{\Omega_L} = \left( (\indic_U E_\omega) * (\indic_\Omega [P(\omega), \chi_\mb] u_\omega)\right)|_{\Omega_L}, \]
    which means $E_\omega * (\indic_\Omega [P(\omega), \chi_\mb] u_\omega) \in \overline C^\infty(\Omega_{L})$ with estimates
    \[\|E_\omega * (\indic_\Omega [P(\omega), \chi_\mb] u_\omega)\|_{H^s(\Omega_{L})} \lesssim \|\widetilde \chi_\mb u_\omega\|_{L^2}.\]
    Restricting to the boundary, taking the differential, and cutting off by $\chi$, we find that
    \[\|\chi r_\omega\|_{H^s} = \|\chi \md\big(((\indic_U E_\omega) * (\indic_\Omega [P(\omega), \chi_\mb] u_\omega))\big|_{\partial \Omega}\big)\|_{H^s} \le \|\widetilde \chi_\mb u_\omega\|_{L^2}\]
    as desired.

    For the differentiated estimate, note that 
    \begin{multline*}
        \partial_\omega r_\omega = (\partial_\omega E_\omega) * (\indic_\Omega [P(\omega), \chi_\mb] u_\omega) \\
        + E_\omega * (\indic_\Omega [\partial_\omega P(\omega), \chi_\mb] u_\omega) + E_\omega * (\indic_\Omega [P(\omega), \chi_\mb] (\partial_\omega u_\omega)).
    \end{multline*}
    Then a similar analysis yields the desired differentiated estimate. 
\end{proof}

Finally, we note that $g_\omega$ is locally smooth, uniformly in $\omega$. In particular, we have the following mapping property. 
\begin{lemma}\label{lem:gw_bound}
    Let $\chi \in \CIc(\partial \Omega)$. Then for any $f \in \CIc(\Omega)$ and $s \ge 0$, 
    \[\|\chi g_\omega\|_{H^{s}(\partial \Omega)} \le C \|f\|_{H^{s + 1}(\Omega)}, \qquad \|\chi \partial_\omega g_\omega\|_{H^{s}(\partial \Omega)} \le C\|f\|_{H^{s + 2}(\Omega)},\]
    where $C$ depends only on $\supp f$, and in particular, the estimate is uniform in $\omega \in \mathcal J + i(0, \infty)$.
\end{lemma}
\begin{proof}
    As in the proof of Lemma~\ref{lem:S_mapping}, we first observe that 
    \[L^\pm_\omega (E_\omega * f) = \kappa_\omega \int_{\R^2} \frac{f(y)}{\ell^\pm(x - y, \omega)}\, \md y.\]
    Using the fact that $\mathcal F(1/z) = -i/(\xi_1 + i \xi_2)$, $z = x_1 + ix_2$, we find that
    \begin{equation*}
        \mathcal F \Big(\frac{1}{\ell^\pm(\bullet, \omega)}\Big) = \Big(\frac{1}{\sqrt{1 -\omega^2}} \xi_1 - \frac{1}{\omega} \xi_2 \Big)^{-1}.
    \end{equation*}
    Then it follows that for any $\psi \in \CIc(\R^2)$,
    \[\|\psi L_\omega^\pm(E_\omega * f)\|_{H^s(\R^2)} \lesssim C \|f\|_{H^s(\R^2)},\]
    where $C$ depends only on $\supp f$. Restricting to the boundary, we conclude that for $s \ge 1$, we have
    \[\|\chi g_\omega\|_{H^{s - 1}(\partial \Omega)} \lesssim C\|f\|_{H^s(\Omega)}.\]
    Next, we can take the derivative with respect to the spectral parameter. Then using the formula~\eqref{eq:dE}, the estimate for $\partial_\omega g_\omega$ follows. 
\end{proof}
\begin{Remark}
    Note that we did not use the fact that our domain is subcritical in the proof above. The estimates can be improved by one order if subcriticality is taken into account. For the sake of concision, we simply assume $f \in \CIc(\Omega)$ and forgo the optimal estimates. We also remark that analogous estimates for higher derivatives in the spectral parameter can be obtained with more work. 
\end{Remark}

\subsection{Kernel computation}\label{sec:kernel}
We now compute the Schwartz kernel of the restricted single layer operator $\md \mathcal C_\omega$ explicitly.
The operator acts on one-forms on $\partial \Omega = \partial \Omega_\uparrow \sqcup \partial \Omega_\downarrow \simeq \R_{\uparrow} \sqcup \R_{\downarrow}$, which we parameterize using the $x_1$ coordinate by 
\begin{equation}\begin{split}\label{param}
& \partial \Omega_\uparrow \coloneqq \{\mathbf x(\theta) \coloneqq (\theta, 0) : \theta\ \in \R_{\uparrow}\}, \\ 
& \partial \Omega_\downarrow \coloneqq  \{\mathbf x(\theta) \coloneqq (\theta, G(\theta) - \pi) : \theta \in \R_{\downarrow}\}.
\end{split}\end{equation}
Note that this parametrization does not agree with a consistent orientation of $\pa\Omega$. Neither it is the parametrization that we used in the proof of Lemma \ref{lem:S_mapping}.
Since we have the explicit formula~\eqref{eq:FS} for the fundamental solution, we see that the Schwartz kernel of $\md \mathcal C_\omega$ can be split into two pieces, $K_\omega = K^+_\omega + K^-_\omega$, given by the formula
\begin{equation}\label{eq:dC_formula}
    K_\omega^\pm (\theta, \theta') = \kappa_\omega \lim_{\delta \to 0 +} \frac{\partial_\theta \ell^\pm (\mathbf x(\theta), \omega)}{\ell^\pm( \mathbf x(\theta) - \mathbf x(\theta') + \delta \mathbf v(\theta),\omega)}
\end{equation}
where $\mathbf v(\theta)$ is an inward pointing vector field. We can simply take 
\[\mathbf v(\theta) = \begin{cases} (0, -1), & \theta \in \partial \Omega_\uparrow, \\ (0, 1), & \theta \in \partial \Omega_\downarrow. \end{cases}\]
Note that each of $K^\pm_\omega$ further consists of four components to account for the choice of $\theta$ and $\theta'$ in $\partial \Omega_\uparrow$ or $\partial \Omega_\downarrow$.  We will analyze these components separately below.

Recall that $\ell^\pm(x, \lambda) = \ell^\pm(\gamma^\pm(x), \lambda)$. We wish to compute the adjustment needed when $\lambda$ is replaced by $\omega = \lambda + i \epsilon$. We have the following
\begin{lemma}\label{lem:invol_errors}
Assume that $\Omega$ is $\lambda$-subcritical. Then for $\omega = \lambda + i \epsilon$, $\epsilon > 0$, there exist $z_\omega^{\pm, \uparrow}(\theta), z_\omega^{\pm, \downarrow}(\theta) \in C^\infty(\R)$ such that 
\begin{gather*}
    \ell^\pm(\mathbf x(\theta), \omega) - \ell^\pm(\gamma^\pm(\mathbf x(\theta)), \omega) = i \epsilon z_\omega^{\pm, \uparrow}(\theta) \quad \text{for} \quad \theta \in \R_\uparrow, \\
    \ell^\pm(\mathbf x(\theta), \omega) - \ell^\pm(\gamma^\pm(\mathbf x(\theta)), \omega) = -i \epsilon z_\omega^{\pm, \downarrow}(\theta) \quad \text{for} \quad \theta \in \R_\downarrow,
\end{gather*}
and $z_\omega^{\pm, \bullet}(\theta)$, $\bullet = \uparrow, \downarrow$, satisfies
\begin{equation}\label{eq:z1}
    C \ge \Re z^{\pm, \bullet}_\omega \ge C^{-1} > 0, \quad \Im z^{\pm, \bullet}_\omega = \mathcal O(\epsilon),
\end{equation}
and there exists $M' > 0$ such that 
\begin{equation}\label{eq:z2}
    \partial_\theta z_\omega^{\pm, \bullet}(\theta) = 0 \quad \text{when} \quad |\theta| > M'.
\end{equation}
\end{lemma}
\begin{proof}
    We compute
    \begin{align*}
        \Lambda(x, \omega) \coloneqq & \ell^\pm(x, \omega) - \ell^\pm(x, \lambda) = \mp \frac{i\epsilon }{(\lambda + i \epsilon) \lambda}x_1 + \frac{\sqrt{1 - \lambda^2} - \sqrt{1 - \omega^2} }{\sqrt{(1 - \lambda^2)(1 - \omega^2)}}x_2 \\
        =& \mp \frac{i \epsilon}{\lambda^2} x_1 +\frac{i\lambda\epsilon}{(1-\lambda^2)^{\frac32}} x_2 + \mathcal O(\epsilon^2).
    \end{align*}
    Note that $\ell^\pm(x, \omega) - \ell^\pm(\gamma^\pm(x), \omega) = \Lambda(x, \omega) - \Lambda(\gamma^\pm(x), \omega)$. When $x \in \partial \Omega_\uparrow$, we see that the $x_1$ coordinate increases and the $x_2$ coordinate decreases when $\gamma^+$ is applied, both by a bounded amount, so 
    \[C \ge \epsilon^{-1} \Im (\ell^+(x, \omega) - \ell^+(\gamma^+(x), \omega)) \ge C^{-1} > 0.\]
    Similarly going through the other cases, we obtain~\eqref{eq:z1}. Finally, \eqref{eq:z2} is an easy consequence of the fact that $G$ is compactly supported. 
\end{proof}

We now assume that $\omega \in (0, 1) + i(0, \infty)$, and analyze the various components of $K_\omega^\pm(\theta,\theta')$ one by one.
For brevity, we slightly overload notation in what follows by writing
$$
\gamma^\pm(\theta)  = x_1(\gamma^\pm(\mathbf x(\theta)))
$$
where $\mathbf x(\theta)$ is one of the parametrizations of a boundary component from \eqref{param}.

\noindent
{\bf Case 1:} $\theta, \theta' \in \R_\uparrow$. Note that for $\epsilon > 0$, $\Im \frac{\omega}{\sqrt{1-\omega^2}} > 0$. Then
\begin{equation*}
    K_\omega^\pm(\theta, \theta') = \kappa_\omega \lim_{\delta \to 0} \frac{\pm \omega^{-1}}{\pm \omega^{-1}(\theta - \theta') - \delta (1 - \omega^2)^{-1/2}} = \kappa_\omega (\theta - \theta' \mp i0)^{-1}.
\end{equation*}
Note that this appears to be the opposite sign as the one appearing in \cite{DyWaZw:21}, but this is accounted for by the convention that $\partial \Omega_\uparrow$ is parameterized in the clockwise direction with respect to the interior of $\Omega$.

\noindent
{\bf Case 2:} $\theta \in \R_\uparrow$, $\theta' \in \R_\downarrow$. Then 
\begin{align*}
    K^\pm_\omega(\theta, \theta') 
    &= \kappa_\omega \frac{\pm \omega^{-1}}{\ell^\pm(\mathbf x(\theta), \omega) - \ell^\pm(\mathbf x(\theta'), \omega)} \\
    &= \kappa_\omega \frac{\pm \omega^{-1}}{[\ell^\pm(\mathbf x(\theta), \omega) - \ell^\pm(\gamma^\pm( \mathbf x(\theta')), \omega)] + [\ell^\pm(\gamma^\pm( \mathbf x(\theta')), \omega) - \ell^\pm(\mathbf x(\theta'), \omega)]} \\
    & = \kappa_\omega \frac{ \pm \omega^{-1} }{ \pm \omega^{-1}(\theta-\gamma^{\pm}(\theta') ) + i\epsilon z_\omega^{\pm, \downarrow}(\theta') } \\
    &= \kappa_\omega \frac{1}{\theta - \gamma^{\pm}(\theta') \pm i \epsilon \omega z^{\pm, \downarrow}_{\omega}(\theta')}
\end{align*}
where $z^{\pm, \downarrow}_{\omega}$ is described in Lemma~\ref{lem:invol_errors}. Note that as $\epsilon \to 0$, the operator wavefront set of $K^+_\omega$ lies uniformly in the {positive} frequencies. In other words, $K^+_\omega$ maps {positive} frequencies on $\partial \Omega_\downarrow$ to {positive} frequencies on $\partial \Omega_\uparrow$ via the map $\gamma^+$, which means {positive} frequencies move to the {left}. On the other hand, we see that $K^-_\omega$ maps {negative} frequencies on $\partial \Omega_\downarrow$ to {negative} frequencies on $\partial \Omega_{\uparrow}$ via the map $\gamma^-$, which means {negative} frequencies move to the {right}. We note that this is consistent with our end analysis since we have high regularity estimates for positive frequencies in the right end of the domain and negative frequencies in the left end of the domain. This means that positive frequencies should propagate left out of the right end, and negative frequencies should propagate right out of the left end.

\noindent
{\bf Case 3:} $\theta \in \R_\downarrow$, $\theta' \in \R_\downarrow$. 
Then
\[K_\omega^\pm(\theta, \theta') = \kappa_\omega \lim_{\delta \to 0+} \frac{\pm \omega^{-1} + G'(\theta) (1 - \omega^2)^{-1/2}}{\pm(\theta - \theta') \omega^{-1} + (G(\theta) - G(\theta'))(1 - \omega^2)^{-1/2} + i \delta}.\]
Note that with $F$ defined as in \eqref{Fdef}, $F(\theta, \theta) = G'(\theta)$. Since $G$ is compactly supported, we also see that there exists $M' > 0$ such that 
\begin{equation}\label{eq:GF_cpct}
    G'(\theta) = 0 \quad \text{and} \quad F(\theta, \theta') = 0 \quad \text{when} \quad |\theta|, |\theta'| > M', \  |\theta - \theta'| \le \frac{1}{M'}.
\end{equation}
By the subcriticality condition, there exists $C > 0$ such that 
\begin{equation}\label{eq:GF_subcrit}
    \begin{gathered}
        C^{-1} \le \pm \Re (\pm \omega^{-1} + G'(\theta) (1 - \omega^2)^{-1/2}) \le C, \\
        C^{-1} \le \pm \Re (\pm \omega^{-1} + F(\theta, \theta') (1 - \omega^2)^{-1/2}) \le C
    \end{gathered}
\end{equation}
for all sufficiently small $\epsilon > 0$. Therefore, taking the $\delta \to 0$ limit yields
\begin{equation}\label{eq:psi_omega}
\begin{gathered}
    K_\omega^\pm (\theta, \theta') = \kappa_\omega \left( \psi_\omega(\theta, \theta') (\theta - \theta')  \pm i 0\right)^{-1}, \\
    \text{where} \quad \psi_\omega(\theta, \theta') \coloneqq \frac{\pm \omega^{-1} + F(\theta, \theta')(1 - \omega^2)^{-1/2}}{\pm \omega^{-1} + G'(\theta) (1 - \omega^2)^{-1/2}}.
\end{gathered}
\end{equation}
Observe that by~\eqref{eq:GF_cpct} and~\eqref{eq:GF_subcrit},  $\psi_\omega \in C^\infty(\R^2)$ and
\begin{equation}\label{eq:psi_prop}
\begin{split}
    & \psi_\omega(\theta, \theta) = 1 \text{ for all } \theta, 
    \\ 
    & \psi_\omega(\theta, \theta') = 1 \text{ when } |\theta|, |\theta'| > M', \  |\theta - \theta'| \le \frac{1}{M'}, \\
    & \Re\psi(\theta,\theta')\in [C^{-1},C] \text{ for all } (\theta,\theta'), \ \epsilon \text{ small, } \text{ and some } C>0, \\
    & |\partial_\theta^\alpha \partial_{\theta'}^\beta \psi_\omega(\theta, \theta')| \le C_{\alpha, \beta}, \text{ for all $\alpha, \beta \in \N_0$},
\end{split}
\end{equation}
where the estimate holds uniformly for all $\epsilon$ sufficiently small. We can write
\begin{equation*}
    K^\pm_\omega(\theta, \theta') = \kappa_\omega (\theta - \theta' \pm i 0)^{-1} + R_\omega^\pm, \quad R_\omega^\pm(\theta, \theta') \coloneqq \kappa_\omega \frac{1 - \psi_\omega(\theta, \theta')}{(\theta - \theta') \psi_\omega(\theta, \theta')}.
\end{equation*}
By~\eqref{eq:psi_prop}, we see that for all $\alpha, \beta\in \mathbb N_0$
\[|\partial_\theta^\alpha \partial_{\theta'}^\beta R_\omega^\pm(\theta, \theta')| \le C_{\alpha, \beta} \langle \theta - \theta' \rangle^{-1}\]
uniform in $\epsilon > 0$ sufficiently small. Therefore, $R_\omega^\pm \in \Psi_{\mathrm{res}}(\R)$.

\noindent
{\bf Case 4:} $\theta \in \R_\downarrow$, $\theta' \in \R_\uparrow$. Then 
\begin{align*}
K^\pm_\omega(\theta, \theta') &= \kappa_\omega \frac{\pm \omega^{-1} + G'(\theta) (1 - \omega^2)^{-1/2}}{\ell^+(\mathbf x(\theta), \omega) - \ell^+(\mathbf x(\theta'), \omega)} \\
&= \kappa_\omega \frac{\pm \omega^{-1} + G'(\theta)(1 - \omega^2)^{-1/2}}{[\ell^\pm(\mathbf x(\theta), \omega) - \ell^\pm(\gamma^\pm( \mathbf x(\theta')), \omega)] + [\ell^\pm(\gamma^\pm( \mathbf x(\theta')), \omega) - \ell^\pm(\mathbf x(\theta'), \omega)]} \\
&= \kappa_\omega \frac{\pm \omega^{-1} + G'(\theta) (1 - \omega^2)^{-1/2}}{ \pm (\theta - \gamma^\pm(\theta')) \omega^{-1} + (G(\theta) - G(\gamma^\pm(\theta')))(1 - \omega^2)^{-1/2} - i\epsilon z_\omega^{\pm, \uparrow}(\theta)}, 
\end{align*}
where $z_\omega^{\pm, \uparrow}$ is described in Lemma~\ref{lem:invol_errors}. Let $\psi_\omega$ be as defined in~\eqref{eq:psi_omega}. Then we see that 
\[K^\pm_\omega(\theta, \theta') = \kappa_\omega \big[\psi_\omega(\theta, \gamma^\pm(\theta')) \cdot (\theta - \gamma^\pm(\theta')) \mp i \epsilon \tilde z_\omega^{\pm, \uparrow}(\theta') \big]^{-1}\]
where
\[\tilde z_\omega^{\pm, \uparrow} = \pm \frac{z_\omega^{\pm, \uparrow}(\theta)}{\pm \omega^{-1} + G'(\theta) (1 - \omega^2)^{-1/2}}.\]
Observe that $\tilde z_\omega^{\pm, \uparrow}$ also satisfies the properties~\eqref{eq:z1} and~\eqref{eq:z2} enjoyed by $z_\omega^{\pm, \uparrow}$ because $\Omega$ is $\lambda$-subcritical. Then we can write
\[K_\omega^\pm(\theta, \theta') = \kappa_\omega \big(\theta - \gamma^\pm(\theta') \mp i \epsilon \tilde z_\omega^{\pm, \uparrow}(\theta') \big)^{-1} + R_\omega^\pm(\theta, \theta')\]
where 
\begin{equation}\label{eq:big_remainder}
    R_\omega^\pm(\theta, \theta') \coloneqq \kappa_\omega \frac{1 - \psi_\omega(\theta, \gamma^\pm(\theta'))}{\psi_\omega(\theta, \gamma^\pm(\theta')) \cdot (\theta - \gamma^\pm(\theta')) \mp i \epsilon \tilde z_\omega^{\pm, \uparrow}(\theta')} \cdot \frac{\theta - \gamma^\pm(\theta')}{\theta - \gamma^\pm(\theta') \mp i \epsilon \tilde z_\omega^{\pm, \uparrow}(\theta')}.
\end{equation}
Observe that by~\eqref{eq:z1} and~\eqref{eq:z2}, 
\[\left|\partial_\theta^\alpha \partial_{\theta'}^\beta \frac{\theta - \gamma^\pm(\theta')}{\theta - \gamma^\pm(\theta') \mp i \epsilon \tilde z_\omega^{\pm, \uparrow}(\theta')} \right| \le C_{\alpha, \beta}.\]
The other term in the remainder~\eqref{eq:big_remainder} can be analyzed using~\eqref{eq:psi_prop} together with~\eqref{eq:z1} and~\eqref{eq:z2}, and one similarly finds
\[\left|\partial_\theta^\alpha \partial_{\theta'}^\beta \frac{1 - \psi_\omega(\theta, \gamma^\pm(\theta'))}{\psi_\omega(\theta, \gamma^\pm(\theta'))\cdot (\theta - \gamma^\pm(\theta')) \mp i \epsilon \tilde z_\omega^{\pm, \uparrow}(\theta')} \right| \le  C_{\alpha, \beta} \langle \theta - \theta' \rangle^{-1}.\]
Therefore, $R_{\omega}^\pm \in \Psi_{\mathrm{res}}(\R)$ uniformly for $\epsilon > 0$ sufficiently small.

Again, this kernel is consistent with the propagation. In particular, $K^+_\omega$ is mapping the negative frequencies on $\Omega_{\uparrow}$ to negative frequencies on $\Omega_{\downarrow}$ via $\gamma^+$, and $K^-_\omega$ is mapping the positive frequencies on $\Omega_{\uparrow}$ to positive frequencies on $\Omega_{\downarrow}$ via $\gamma^-$. So we again see that positive frequencies propagate to the left and negative frequencies propagate to the right when $\epsilon > 0$.

We collect the above kernel computations in the following proposition and compute the symbols. Using the coordinates $\theta$, we identify (say smooth and compactly supported) 1-forms on the boundary $\partial \Omega$ with the pair $(v_\uparrow(\theta), v_\downarrow(\theta)) \md \theta$ where $v_\bullet \in \CIc(\R_\bullet)$, $\bullet = \uparrow, \downarrow$. 
\begin{proposition}\label{prop:kernel}
The operator $\md \mathcal C_\omega$ is given by 
    \[\md \mathcal C_\omega = \begin{pmatrix} \mathscr H & B^{+, \uparrow}_{\omega} (\gamma^+)^*  + B^{-, \uparrow}_{\omega} (\gamma^-)^* \\ B^{-, \downarrow}_{\omega} (\gamma^+)^*  + B^{+, \downarrow}_{\omega} (\gamma^-)^* & \mathscr H\end{pmatrix} + \mathcal R_\omega\]
    where $(\gamma^\pm)^*$ act as pull-backs on 1-forms, $\mathscr H \coloneqq -2 \pi i \kappa_\omega \Op(\sgn(\xi))$ denotes a multiple of the Hilbert transform, and for $\bullet=\uparrow, \downarrow$,
    \[B_\omega^{\pm, \bullet} = \Op(b_\omega^{\pm, \bullet}), \qquad b_\omega^{\pm, \bullet}(\theta', \xi) = \mp 2 \pi i \kappa_\omega H(\pm \xi) e^{-\epsilon z_\omega^{\pm, \bullet}(\theta')|\xi|}\]
    where $z^{\pm, \bullet}_\omega$ satisfies~\eqref{eq:z1} and~\eqref{eq:z2}, and every entry of $\mathcal R_\omega$ lies in $\Psi_{\mathrm{res}}(\R)$. 
\end{proposition}
\begin{Remarks}
1. The convention here is that $B^{\pm, \bullet}_\omega$ has operator wavefront set on the positive/negative frequencies, and the $\bullet$ indicates the component of boundary on which the operator acts. 

\noindent
2. By a slight abuse of notation, $z_\omega^{\pm, \bullet}$ may differ from those defined in Lemma~\ref{lem:invol_errors}, but they satisfy the same properties as in the lemma. 
\end{Remarks}
\begin{proof}
    The four entries of $\md \mathcal C_\omega$ correspond to the four cases above, so we compute the symbols accordingly. 

    \noindent
    {\bf Case 1:} $\theta, \theta' \in \R_{\uparrow}$. Note that 
    \[K_\omega^\pm = \kappa_\omega (\theta - \theta' \mp i0)^{-1} = \pm i \kappa_\omega \int e^{i (\theta - \theta') \xi} H(\mp \xi) \, \md \xi .\]
    So it follows that the corresponding entry is given by $- 2 \pi i \Op(H(\mp \xi))$. 

    \noindent
    {\bf Case 2:} $\theta \in \R_\uparrow$, $\theta' \in \R_{\downarrow}$. We see that 
    \[K_\omega^\pm(\theta, \gamma^{\pm}(\theta')) = \mp i\kappa_\omega \int e^{i (\theta - \theta') \xi} H(\pm \xi) e^{- \epsilon |\xi|z_\omega^{\pm, \downarrow}(\gamma^\pm(\theta'))} \, \md \xi.\]
    Note that $z^{\pm, \downarrow}_{\omega}(\gamma^\pm(\theta'))$ satisfies~\eqref{eq:z1} and~\eqref{eq:z2}. Then upon redefining $z^{\pm, \uparrow}_{\omega}(\theta') = z^{\pm, \downarrow}_{\omega}(\gamma^\pm(\theta'))$, it follows that the corresponding entry is given by 
    \[K^+_\omega + K^-_\omega = B^{+, \uparrow}_\omega (\gamma^+)^* + B^{-, \uparrow}_\omega (\gamma^-)^*. \]
    where $(\gamma^\pm)^*$ denotes the pullback on 1-forms. 

    For the remaining two cases, the analysis for {\bf Case 3} is nearly identical to {\bf Case 1} and the analysis for {\bf Case 4} is nearly identical to {\bf Case 2}. 
\end{proof}

The following mapping property follows immediately from Proposition~\ref{prop:kernel}.
\begin{corollary}\label{cor:diffC}
    Let $\omega \in \mathcal J + i(0, \infty)$. Then 
    \[\partial_\omega^k \md\mathcal C_\omega : H^{s + k}_{\comp}(\partial \Omega; T^* \partial \Omega) \to H^{s}_{\loc}(\partial \Omega; T^* \partial \Omega)\]
    uniformly in $\omega$. 
\end{corollary}
The higher derivatives in the spectral parameter are again needed to establish regularity of the spectral measure.

\section{Global high frequency estimate}\label{sec:global}
The goal of this section is to prove the semi-Fredholm estimate found in Proposition~\ref{prop:SFE}. To do so, we link the high frequency decay estimates of Proposition~\ref{prop:high_decay} to the low frequency decay estimates of Proposition~\ref{prop:low_decay} by propagating through the black box region. We will use the black box construction and notations found in~\eqref{eq:bb1}--\eqref{eq:bb3} throughout this section.

\subsection{Propagation inside the black box}
We proceed to establish propagation estimates for the boundary reduced problem inside the black box. For the sake of clarity, we focus our attention to how positive frequencies propagate. The analysis for negative frequencies is nearly identical with propagation in the opposite direction to the positive frequencies. Let 
\[\widetilde \Pi_+ \coloneqq \begin{pmatrix} \Op(h_+(\xi)) & 0 \\ 0 & \Op(h_+(\xi)) \end{pmatrix}\]
where $h_+ \in C^\infty(\R; \R)$, $h_+(\xi) = 1$ for $\xi \ge 1$ and $h_+(\xi) = 0$ for $\xi \le 1/2$. By a slight abuse of notation, we will also write $\widetilde \Pi_+ = \Op(h_+(\xi))$, since it will be clear from context what space $\widetilde \Pi_+$ acts on. It follows from the quantization formula that 
\[\widetilde \Pi_+ B^{-, \bullet}_\omega = 0, \quad \bullet = \uparrow,\, \downarrow.\]
Then
\begin{equation}\label{eq:pdC0}
    \widetilde \Pi_+ \md \mathcal C_\omega = - 2 \pi i \kappa_\omega \widetilde \Pi_+ + \begin{pmatrix} 0 & \widetilde \Pi_+ B^{+, \uparrow}_\omega (\gamma^+)^* \\ \widetilde \Pi_+ B^{+, \downarrow}_\omega (\gamma^-)^* & 0\end{pmatrix} + \mathcal R_\omega
\end{equation}
where $\mathcal R_\omega$ may change from line to line, but each component of $\mathcal R_\omega$ always lies uniformly in $\Psi_{\mathrm{res}}(\R)$ for $\epsilon > 0$ sufficiently small. Note that by~\eqref{eq:cov}, we have
\[(\gamma^\pm)^*\widetilde \Pi_+ (\gamma^\pm)^* - \widetilde \Pi_+ \in \Psi^{-\infty, -\infty}(\R) \subset \Psi_{\mathrm{res}}(\R),\]
and by the composition formula~\eqref{eq:composition} and Lemma~\ref{lem:cutoff_decomp}, we see that 
\[[\widetilde \Pi_+, B_\omega^{+, \bullet}] \in \Psi^{-\infty, -\infty}(\R) + \Psi_{\mathrm{res}}(\R) \subset \Psi_{\mathrm{res}}(\R).\]
Therefore, it follows that 
\begin{align*}
    \widetilde \Pi_+ B_\omega^{+, \bullet} (\gamma^\pm)^* &= [\widetilde \Pi_+, B_\omega^{+, \bullet}] (\gamma^\pm)^* + B_\omega^{+, \bullet} \widetilde \Pi_+ (\gamma^\pm)^* \\
    &= [\widetilde \Pi_+, B_\omega^{+, \bullet}] (\gamma^\pm)^* + B_\omega^{+, \bullet}(\gamma^\pm)^* \widetilde \Pi_+ + B_\omega^{+, \bullet}(\gamma^\pm)^*((\gamma^\pm)^* \widetilde \Pi_+ (\gamma^\pm)^* - \widetilde \Pi_+) \\
    &= B_\omega^{+, \bullet}(\gamma^\pm)^* \widetilde \Pi_+ + \mathcal R_\omega.
\end{align*}
Then we can rearrange~\eqref{eq:pdC0} to 
\begin{equation}\label{eq:pdC}
    \widetilde \Pi_+ \md \mathcal C_\omega = \left[ -2 \pi i \kappa_\omega \Id  + \begin{pmatrix} 0 & \widetilde \Pi_+ B^{+, \uparrow}_\omega (\gamma^+)^* \\ \widetilde \Pi_+ B^{+, \downarrow}_\omega (\gamma^-)^* & 0\end{pmatrix}\right] \widetilde \Pi_+ + \mathcal R_\omega.
\end{equation}
Put
\begin{equation}\label{eq:tilde_v}
    \tilde v_\omega^+ \coloneqq \widetilde \Pi_+ \chi_\mb v_\omega
\end{equation}
where recall that $\chi_\mb$ is the black box cutoff. By the composition formula~\eqref{eq:composition}, for all $\chi\in C_c^{\infty}(\RR)$, $[\widetilde \Pi_+, \chi] \in \Psi^{-\infty, -\infty}(\R) \subset \Psi_{\mathrm{res}}(\R)$. Then by~\eqref{eq:bdr} and~\eqref{eq:pdC}, we see that for every $\chi \in \CIc(\partial \Omega)$, 
\begin{equation}\label{eq:cohom}
    \begin{aligned}
        \widetilde \Pi_+ \chi (r_\omega + g_\omega) &= \widetilde \Pi_+ \chi \md \mathcal C_\omega (\chi_\mb v_\omega) \\
        &= \chi \widetilde \Pi_+ \md \mathcal C_\omega (\chi_\mb v_\omega) + \mathcal R_\omega \chi_\mb v_\omega \\
        &= - 2 \pi i \kappa_\omega \chi \tilde v_\omega^+ + \begin{pmatrix} 0 & \chi B^{+, \uparrow}_\omega (\gamma^+)^* \\ \chi B^{+, \downarrow}_\omega (\gamma^-)^* & 0\end{pmatrix} \tilde v_\omega^+ + \mathcal R_\omega \chi_\mb v_\omega.
    \end{aligned}
\end{equation}
It is helpful to interpret~\eqref{eq:cohom} as a microlocal cohomological equation (see for instance \cite[(2.10)]{LiWaWu:24} and \cite[(2.6)]{CoLi:24}) --- without the cutoff $\chi$ and the operators $B^{+, \bullet}_\omega$, the right-hand-side of~\eqref{eq:cohom} is simply identity plus pull-back.

\begin{lemma}\label{lem:bb_prop}
    Let $\chi = (\chi_{\uparrow}, \chi_{\downarrow}) \in \CIc(\partial \Omega)$ be such that $\supp \chi \subset \{|x_1| < L\}$. For $\bullet = \uparrow$, $\downarrow$, or $\mb$, let $\widetilde \chi_\bullet \in \CIc(\partial \Omega)$ be such that $\widetilde \chi_\bullet = 1$ on $\supp \chi_\bullet$. Then 
    \begin{equation*}\begin{split}
        \|\widetilde \Pi_\pm \chi v_\omega \|_{H^s(\partial \Omega)} \le C \big( & \|\widetilde \Pi_\pm(\widetilde \chi_\uparrow \circ \gamma^\pm) v_\omega\|_{H^s(\partial \Omega_\downarrow)} + \|\widetilde \Pi_\pm(\widetilde \chi_\downarrow \circ \gamma^\mp) v_\omega\|_{H^s(\partial \Omega_\uparrow)} \\
        & + \|\widetilde \chi_\mb u_\omega\|_{\bar H^{2}(\Omega)} + \|f\|_{H^{s + 1}(\Omega)} \big).
    \end{split}\end{equation*}
\end{lemma}
\begin{proof}
    We focus on the $\widetilde \Pi_+$ case since the $\widetilde \Pi_-$ case follows via the same argument with appropriate sign changes. We further assume that $\chi_\downarrow = 0$. Then using \eqref{eq:cohom}, we have
    \begin{equation}\begin{split}\label{eq:cohom_est}
        \|\chi_\uparrow \tilde v_\omega^+ \|_{H^s(\partial \Omega_\uparrow)} \lesssim 
        & \|\chi_\uparrow B_\omega^{+, \uparrow} (\gamma^+)^* \tilde v_\omega^+\|_{H^s(\partial \Omega_\uparrow)} \\
        & + \|\chi_\uparrow \mathcal R_\omega \chi_\mb v_\omega\|_{H^s(\partial\Omega_{\uparrow})} + \| \chi_\uparrow \widetilde \Pi_+(r_\omega + g_\omega))\|_{H^s(\partial \Omega_\uparrow)}.
    \end{split}\end{equation}
    First, we see that
    \begin{equation}\label{eq:rv_est}
        \|\chi_\uparrow \mathcal R_\omega \chi_\mb v_\omega\|_{H^s(\partial\Omega_{\uparrow})} \lesssim \|\chi_\mb v_\omega\|_{H^{-N}(\partial\Omega_{\uparrow})}.
    \end{equation}
    By Lemma~\ref{lem:rw_est} and~\ref{lem:gw_bound}, we also have 
    \begin{equation}\label{eq:r+g_est}
        \|\widetilde \Pi_+ \chi_\uparrow(r_\omega + g_\omega)\|_{H^s(\partial \Omega_\uparrow)} \lesssim \|f\|_{H^{s + 1}(\Omega)} + \|\widetilde \chi_\mb u_\omega\|_{L^2(\Omega)}.
    \end{equation}
    Next, observe that 
    \[\chi_\uparrow B_\omega^{+, \uparrow} (\gamma^+)^* = (\gamma^+)^*(\chi_\uparrow \circ \gamma^+) \widetilde B_\omega^{+, \uparrow} = (\gamma^+)^*(\chi_\uparrow \circ \gamma^+) \widetilde B_\omega^{+, \uparrow} (\widetilde \chi_\uparrow \circ \gamma^+) + \chi_\mb \mathcal R_{\omega}\]
    where $\widetilde B_\omega^{+, \uparrow} \coloneqq (\gamma^+)^* B_\omega^{+, \uparrow} (\gamma^+)^*$, and $\mathcal R_{\omega}$ is some operator in $\Psi_{\mathrm{res}}(\R)$. Therefore,  
    \begin{equation}\label{eq:transport_est}
        \|\chi_\uparrow B_\omega^{+, \uparrow} (\gamma^+)^* \tilde v_\omega^+\|_{H^s(\partial \Omega_\uparrow)} \lesssim \|(\widetilde \chi_\uparrow \circ \gamma^+) \tilde v_\omega^+\|_{H^s(\partial \Omega_\downarrow)} + \|\chi_\mb v_\omega\|_{H^{-N}(\partial\Omega)}.
    \end{equation}
    Note that since $v_\omega$ is the Neumann data of $u_\omega$,
    \begin{equation}\label{residualbound}
    \|\chi_\mb v_\omega\|_{H^{-N}(\partial\Omega)} \lesssim \|\chi_\mb u_\omega\|_{\bar H^2(\Omega)}.
    \end{equation}
    Then bounding the right-hand-side of~\eqref{eq:cohom_est} by~\eqref{eq:rv_est}--\eqref{eq:transport_est}, we find that 
    \begin{equation}\label{eq:tilde_v_est1}
        \|\chi_\uparrow \tilde v_\omega^+ \|_{H^s(\partial \Omega_\uparrow)} \lesssim \|(\widetilde \chi_\uparrow \circ \gamma^+) \tilde v_\omega^+\|_{H^s(\partial \Omega_\downarrow)} + \|\widetilde \chi_\mb u_\omega\|_{H^{2}(\Omega)} + \|f\|_{H^{s + 1}(\Omega)}.
    \end{equation}
    By our assumption, $\chi_\mb = 1$ on $\supp \chi_\uparrow$, and we may further assume without the loss of generality that $\chi_\mb = 1$ on $\supp \widetilde \chi_\uparrow \circ \gamma^+$. Note that such $\widetilde \chi_\uparrow$ always exists by our black box construction~\eqref{eq:bb1}--\eqref{eq:bb3}. Then we see that
    \[\chi_\uparrow \widetilde \Pi_+ \chi_\mb - \widetilde \Pi_+ \chi_\uparrow \in \Psi^{-\infty, -\infty}(\R), \quad (\widetilde \chi_\uparrow \circ \gamma^+) \widetilde \Pi_+ \chi_\mb - \widetilde \Pi_+ (\widetilde \chi_\uparrow \circ \gamma^+) \in \Psi^{-\infty, -\infty}(\R),\]
    Unwinding the definition of $\tilde v^+_\omega$ given in~\eqref{eq:tilde_v} and continuing to bound residual term on the boundary by using \eqref{residualbound}, it follows from~\eqref{eq:tilde_v_est1} that
    \[\|\widetilde \Pi_+ \chi_\uparrow v_\omega \|_{H^s(\partial \Omega_\uparrow)} \lesssim \|\widetilde \Pi_+ (\widetilde \chi_\uparrow \circ \gamma^+) v_\omega\|_{H^s(\partial \Omega_\downarrow)} + \|\widetilde \chi_\mb u_\omega\|_{\bar H^{2}(\Omega)} + \|f\|_{H^{s + 1}}.\]
    The case that $\chi_{\downarrow}$ is nontrivial and $\chi_{\uparrow}$ is trivial follows via a similar analysis with minor sign changes. 
\end{proof}

\subsection{Global propagation}
Now we iterate the above estimates to control the interior region $\{|x_1| \le M\}$ by the ends.
\begin{lemma}\label{lem:interior_prop}
    Let $\chi_{\mathrm{int}}, \chi_{\mathrm{R}}, \chi_{\mathrm{L}} \in \CIc(\partial \Omega)$ be such that 
    \begin{equation*}
    \begin{gathered}
        \chi_{\mathrm{int}} \equiv 1 \quad \text{for}  \quad |x_1| \le 2M, \quad \supp \chi_{\mathrm{int}} \subset \{|x_1| \le 4M\}, \\
        \chi_{\mathrm{R}} \equiv 1 \quad \text{for} \quad x_1 \in [4M, L], \quad \supp \chi_{\mathrm{R}} \subset \{|x_1|\le 2L\}, \\
        \chi_{\mathrm{L}}(x_1) = \chi_{\mathrm{R}}(-x_1).
    \end{gathered}
    \end{equation*}
    Then 
    \[\|\chi_{\mathrm{int}} v_\omega\|_{H^s(\partial \Omega)} \le C \big(\|\widetilde \Pi_+ \chi_{\mathrm{R}} v_\omega\|_{H^s(\partial \Omega)} + \|\widetilde \Pi_- \chi_{\mathrm{L}} v_\omega\|_{H^s(\partial \Omega)} + \|\widetilde \chi_\mb u_\omega\|_{\bar H^{2}(\Omega)} + \|f\|_{H^{s + 1}(\Omega)} \big).\]
\end{lemma}
\begin{proof}
    We first focus on the positive frequencies on the top piece of boundary, and consider the cutoffs $\chi_\bullet$ as elements of $\CIc(\partial\Omega_\uparrow)$. Then applying Lemma~\ref{lem:bb_prop} twice, we find that for any $\widetilde \chi_{\mathrm{int}} \in \CIc(\partial \Omega_\uparrow)$ such that $\widetilde \chi_{\mathrm{int}} \equiv 1$ on $\supp \chi_{\mathrm{int}}$, we have
    \begin{equation}\label{eq:twice}
        \|\widetilde \Pi_+ \chi_{\mathrm{int}} v_\omega\|_{H^s(\partial \Omega_{\uparrow})} \lesssim \|\widetilde \Pi_+(\widetilde \chi_{\mathrm{int}} \circ b^{-1}) v_\omega \|_{H^s(\partial \Omega_{\uparrow})} + \|\widetilde \chi_\mb u_\omega\|_{\bar H^{2}(\Omega)} +  \|\widetilde \chi_\mb g_\omega\|_{H^{s}(\partial \Omega)}.
    \end{equation}
    Let $N > 0$ be so that~\eqref{eq:bb2} holds. Then it follows that 
    \[\supp (\chi_{\mathrm{int}} \circ b^{-N}) \subset \{ \chi_{\mathrm R}=1 \}\cap \partial\Omega_{\uparrow}.\]
    Therefore, we can arrange $\widetilde \chi_{\mathrm{int}}$ in \eqref{eq:twice} so that $\widetilde \chi_{\mathrm{int}} \circ b^{-N} \le \chi_{\mathrm{R}}$ on $\partial\Omega_{\uparrow}$. Then iterating~\eqref{eq:twice} $N$-times and using Lemma~\ref{lem:gw_bound}, we conclude that 
    \begin{equation*}
        \|\widetilde \Pi_+ \chi_{\mathrm{int}} v_\omega\|_{H^s(\partial \Omega_{\uparrow})} \le C\big(\|\widetilde \Pi_+ \chi_{\mathrm{R}} v_\omega \|_{H^s(\partial \Omega_{\uparrow})} + \|\widetilde \chi_\mb u_\omega\|_{\bar H^{2}(\Omega)} +  \|f\|_{H^{s + 1}(\Omega)}\big).
    \end{equation*}
    The case for positive frequencies on $\partial \Omega_\downarrow$ also follows similarly from iterating Lemma~\ref{lem:bb_prop} with minor sign changes. For negative frequencies, one can iterate the $\widetilde \Pi_-$ sign in Lemma~\ref{lem:bb_prop}. Summing the positive and negative frequencies yields the desired result.
\end{proof}

For higher regularity of the spectral measure, we also need analogous estimates for $\partial_\omega v_\omega$. We have the following:
\begin{lemma}\label{lem:interior_prop_diff}
Let $\chi_\bullet$ be the same cutoffs as in Lemma~\ref{lem:interior_prop}. Then
    \begin{equation*}\begin{split}
        \|\chi_{\mathrm{int}} \partial_\omega v_\omega \|_{H^s(\partial \Omega)} 
        \le C \big( & \|\widetilde \Pi_+ \chi_{\mathrm{R}} \partial_\omega v_\omega \|_{H^s(\partial \Omega)} + \|\widetilde \Pi_- \chi_{\mathrm{L}} \partial_{\omega} v_\omega \|_{H^s(\partial \Omega)} \\
        & + \|\widetilde \Pi_+ \chi_{\mathrm{R}} v_\omega \|_{H^{s + 1}(\partial \Omega)}
        + \|\widetilde \Pi_- \chi_{\mathrm{L}} v_\omega \|_{H^{s + 1}(\partial \Omega)} \\
        &+ \|\widetilde \chi_\mb u_\omega\|_{\bar H^{2}(\Omega)} + \|\widetilde \chi_\mb \partial_\omega u_\omega\|_{\bar H^{2}(\Omega)} + \|f\|_{H^{s + 2}(\Omega)}\big).
    \end{split}\end{equation*}
\end{lemma}
\begin{proof}

Differentiating~\eqref{eq:bdr} in $\omega$, we find that
\begin{equation}\label{eq:diff_bdr}
    \md \mathcal C_\omega (\chi_\mb \partial_\omega v_\omega) = \partial_\omega g_\omega + \partial_\omega r_\omega - (\partial_\omega \md \mathcal C_\omega)(\chi_\mb v_\omega).
\end{equation}
Note that the equation~\eqref{eq:diff_bdr} for $\partial_\omega v_\omega$ only differs from the equation~\eqref{eq:bdr} for $v_\omega$ on the right-hand-side, and an analogous computation to~\eqref{eq:cohom} yields
\begin{multline}\label{eq:cohom_diff}
    \widetilde \Pi_\pm \chi \big[\partial_\omega g_\omega + \partial_\omega r_\omega - (\partial_\omega \md \mathcal C_\omega)(\chi_\mb v_\omega)\big] \\
    = \mp 2 \pi i \kappa_\omega \chi \widetilde \Pi_\pm \chi_\mb \partial_\omega v_\omega + \begin{pmatrix} 0 & \chi B^{\pm, \downarrow}_\omega (\gamma^\pm)^* \\ \chi B^{\pm, \uparrow}_\omega (\gamma^\mp)^* & 0\end{pmatrix} \partial_\omega \widetilde \Pi_\pm \chi_\mb v_\omega + \mathcal R_\omega \chi_\mb \partial_\omega v_\omega
\end{multline}
for any $\chi = (\chi_\uparrow, \chi_\downarrow) \in \CIc(\partial \Omega)$. Further assume that
\begin{equation}\label{eq:chi_supp}
    \supp \chi \subset [-L, L], \qquad \supp (\chi \circ \gamma^\pm) \subset [-L, L],
\end{equation}
and let $\widetilde \chi_\bullet$ be smooth and compactly supported cutoffs such that $\widetilde \chi_\bullet = 1$ on $\supp \chi_\bullet$. Then taking the $H^s$-norm on both sides of \eqref{eq:cohom_diff} and proceeding as in Lemma~\ref{lem:bb_prop}, we find 
\begin{equation}\begin{split}\label{eq:bb_prop_diff1}
    \|\widetilde \Pi_\pm \chi \partial_\omega v_\omega\|_{H^s(\partial \Omega)} 
    \le C \big( & \|\widetilde \Pi_\pm (\widetilde \chi_\uparrow \circ \gamma^\pm) \partial_\omega v_\omega \|_{H^s(\partial \Omega_\downarrow)} + \|\widetilde \Pi_\pm (\widetilde \chi_\downarrow \circ \gamma^\mp) v_\omega \|_{H^s(\partial \Omega_\uparrow)} \\
    & + \|\widetilde \chi_\mb u_\omega \|_{\bar H^2(\Omega)} + \|\widetilde \chi_\mb \partial_\omega u_\omega \|_{\bar H^2(\Omega)} \\
    & + \|f\|_{H^{s + 2}(\Omega)} + \|\widetilde \Pi_\pm \chi (\partial_\omega \md \mathcal C_\omega)(\chi_\mb v_\omega)\|_{H^s(\partial \Omega)} \big),
\end{split}\end{equation}
where we used the differentiated estimates in Lemmas~\ref{lem:gw_bound} and~\ref{lem:rw_est} to control the $\partial_\omega g_\omega$ and $\partial_\omega r_\omega$ terms. Now we control the last term on the right-hand-side of~\eqref{eq:bb_prop_diff1}.
We may choose $\widetilde \chi$ to also satisfy~\eqref{eq:chi_supp}. Then
\[\begin{split} 
& \|\widetilde \Pi_\pm \chi_\uparrow (\partial_\omega \md \mathcal C_\omega)(\chi_\mb v_\omega)\|_{H^s(\partial \Omega)} \\
    & \ \le \|\widetilde \Pi_\pm \chi_{\uparrow} (\partial_\omega \md \mathcal C_\omega)((\widetilde \chi \circ \gamma^\mp) v_\omega)\|_{H^s(\partial \Omega)}
    + \|\widetilde \Pi_\pm \chi_\uparrow (\partial_\omega \md \mathcal C_\omega)((\chi_\mb  -\widetilde \chi \circ \gamma^\mp) v_\omega)\|_{H^s(\partial \Omega)}. 
\end{split}\]
We use Corollary~\ref{cor:diffC} to control the first term on the right-hand-side 
\[ \|\widetilde \Pi_{\pm} \chi_{\uparrow}(\partial_\omega \md \mathcal C_\omega)((\widetilde \chi \circ \gamma^{\mp}) v_\omega) \|_{H^s(\partial\Omega)} \leq C\| (\widetilde \chi \circ \gamma^{\mp}) v_\omega \|_{H^{s+1}(\partial\Omega)}. \]
For the second term on the right-hand-side, notice that we can choose $\widetilde \chi$ such that $\supp \chi_{\uparrow}$ does not intersect with any of $\supp (\chi_\mb - \widetilde \chi \circ \gamma^\mp)$, $\gamma^{\pm}(\supp (\chi_\mathrm b - \widetilde \chi \circ \gamma^\mp))$. Hence by Proposition \ref{prop:kernel}, 
\[\widetilde \Pi_\pm \chi_{\uparrow}(\partial_\omega \md \mathcal C_\omega)(\chi_\mb - \widetilde \chi\circ \gamma^{\pm})\in \Psi^{-\infty, -\infty}(\partial\Omega).\] 
As a result, also using the fact that $v_\omega$ is the Neumann data of $u_\omega$,
\[ \|\widetilde \Pi_\pm \chi_\uparrow (\partial_\omega \md \mathcal C_\omega)((\chi_\mb  -\widetilde \chi \circ \gamma^\mp) v_\omega)\|_{H^s(\partial \Omega)} \leq C\|\widetilde \chi_\mb v_\omega\|_{H^{-N}(\partial\Omega)}\leq C\|\widetilde \chi_\mb u_\omega \|_{\bar H^2(\Omega)}. \]
Here we regard $\widetilde \chi_\mb$ as cutoff functions on both $\partial\Omega$ and $\Omega$.
Combining estimates for both terms on the right-hand-side, we find
\begin{equation*}\begin{split}
    & \|\widetilde \Pi_\pm \chi_\uparrow (\partial_\omega \md \mathcal C_\omega)(\chi_\mb v_\omega)\|_{H^s(\partial \Omega)} \\
    & \ \lesssim \|(\widetilde \chi \circ \gamma^\mp) v_\omega\|_{H^{s + 1}(\partial \Omega)} + \|\widetilde \chi_\mb u_\omega \|_{\bar H^{2}(\Omega)} \\
    & \ \lesssim \|\widetilde \Pi_+ \chi_{\mathrm{R}} v_\omega\|_{H^{s + 1}(\partial \Omega)} + \|\widetilde \Pi_- \chi_{\mathrm{L}} v_\omega\|_{H^{s + 1}(\partial \Omega)} 
    + \|\widetilde \chi_\mb u_\omega\|_{\bar H^{2}(\Omega)} + \|f\|_{H^{s + 2}(\Omega)},
\end{split}\end{equation*}
where the last inequality is due to Lemma~\ref{lem:interior_prop}. 
We also have the identical inequality when $\chi_\uparrow$ is replaced by $\chi_\downarrow$. Then combining with~\eqref{eq:bb_prop_diff1}, we find that 
\begin{equation}\begin{split}\label{eq:bb_prop_diff2}
    \|\widetilde \Pi_\pm \chi \partial_\omega v_\omega\|_{H^s(\partial \Omega)} 
    \le C \big( & \|\widetilde \Pi_\pm (\widetilde \chi_\uparrow \circ \gamma^\pm) \partial_\omega v_\omega \|_{H^s(\partial \Omega)} + \|\widetilde \Pi_\pm (\widetilde \chi_\downarrow \circ \gamma^\mp) v_\omega \|_{H^{s}(\partial \Omega)} \\
    & + \|\widetilde \Pi_+ \chi_{\mathrm{R}} v_\omega\|_{H^{s + 1}(\partial \Omega)} + \|\widetilde \Pi_- \chi_{\mathrm{L}} v_\omega\|_{H^{s + 1}(\partial \Omega)}
    \\
    & + \|\widetilde \chi_\mb u_\omega \|_{\bar H^2(\Omega)} + \|\widetilde \chi_\mb \partial_\omega u_\omega \|_{\bar H^2(\Omega)} + \|f\|_{H^{s + 2}(\Omega)} \big)
\end{split}\end{equation}
Iterating the estimate~\eqref{eq:bb_prop_diff2} similar to Lemma~\ref{lem:interior_prop} yields the desired estimate. 
\end{proof}

Now assembling the propagation estimate with the end decay estimates, we conclude the desired semi-Fredholm estimate. 
\begin{proposition}\label{prop:SFE}
Assume that $\omega \in \mathcal J + i(0, \infty)$. Let $u_\omega \in H^1_0(\Omega)$ be the unique solution to~\eqref{eq:P(w)u}. Then for any $\beta < -1/2$, $s \in \R$, and $N \in \NN$,
    \[\|u_\omega\|_{\bar H^{s, \beta}(\Omega)} \le C\big(\|f\|_{H^{s}(\Omega)} + \|u_\omega\|_{H^{2, -N}(\Omega)}\big)\]
    uniformly in $\omega = \lambda + i \epsilon$ for all sufficiently small $\epsilon > 0$. 
\end{proposition}
\begin{proof}
    1. We first control $u_\omega$ in the region $\{|x_1| \le M\}$ by $v_\omega$. Let $\chi \in \CIc(\R)$ be such that $\chi = 1$ on $[-M, M]$ and $\supp \chi \subset (-\frac{3}{2} M, \frac{3}{2} M)$. Recall $\chi_{\mathrm{int}} \in \CIc(\R_{x_1})$ defined in Lemma~\ref{lem:interior_prop}. Since $E_\omega$ is uniformly smooth in every compact set supported away from $\{x \in \R^2 \mid \ell^\pm(x, \lambda) = 0, \ \lambda \in \mathcal J\}$, it follows that for every $N \in \R$, 
    \[\|\chi S_\omega (\chi_\mb - \chi_{\mathrm{int}})v_\omega \|_{\bar H^s(\Omega)} = \|\chi E_\omega * \mathcal I \big((\chi_\mb - \chi_{\mathrm{int}}) v_\omega\big)\|_{\bar H^s(\Omega)} \lesssim \|\chi_b v_\omega\|_{H^{-N}(\partial \Omega)}.\]
    Then by Lemma~\ref{lem:S_mapping}, we have
    \begin{align*}
        \|\chi S_\omega \chi_\mb v_\omega\|_{\bar H^{s}(\Omega)} &\le \|\chi S_\omega \chi_{\mathrm{int}} v_\omega \|_{\bar H^{s}(\Omega)} + \|\chi S_\omega (\chi_\mb - \chi_{\mathrm{int}})v_\omega\|_{\bar H^{s}(\Omega)} \\
        &\lesssim \|\chi_{\mathrm{int}} v_\omega\|_{H^{s - 1}(\partial \Omega)} + \|\chi_\mb v_\omega\|_{H^{-N}(\partial \Omega)} \\
        &\lesssim \|\chi_{\mathrm{int}} v_\omega\|_{H^{s - 1}(\partial \Omega)} + \|u_\omega\|_{H^{2, -N}(\Omega)}.
    \end{align*}
    Then by~\eqref{eq:single_layer} and Lemma~\ref{lem:interior_prop},
    \begin{align*}
        \|\chi u_\omega \|_{\bar H^{s}} &\lesssim \|\chi_{\mathrm{int}} v_\omega\|_{H^{s - 1}(\partial \Omega)} + \|u_\omega\|_{H^{2, -N}(\Omega)} + \|f\|_{H^{s - 1}(\Omega)} \\
        &\lesssim \|\widetilde \Pi_+ \chi_{\mathrm{R}} v_\omega\|_{H^{s - 1}(\partial \Omega)} + \|\widetilde \Pi_- \chi_{\mathrm{L}} v_\omega\|_{H^{s - 1}(\partial \Omega)} + \|u_\omega\|_{\bar H^{2, -N}(\Omega)} + \|f\|_{H^{s}(\Omega)}.
    \end{align*}

    \noindent
    2. By the high frequency decay estimate of Proposition~\ref{prop:high_decay}, we then find that
    \begin{equation}\label{eq:fred_interior}
        \|\chi u_\omega\|_{\bar H^{s}(\Omega)} \lesssim \|u_\omega\|_{\bar H^{2, -N}(\Omega)} + \|f\|_{H^{s}(\Omega)}.
    \end{equation}
    Notice that the $\|\chi_\pm v_\omega\|_{H^{-N,-N}}$ terms in Proposition~\ref{prop:high_decay} are absorbed in $\|u_\omega\|_{\bar H^{2,-N}(\Omega)}$ here.
    Note that $1 - \chi$ cuts off to the left and right ends of the domain. Therefore, by the low frequency decay estimates of Proposition~\ref{prop:low_decay}, we see that for any $\beta < -1/2$, 
    \begin{equation}\label{eq:fred_ends}
        \|(1 - \chi)u_\omega\|_{\bar H^{s, \beta}(\Omega)} \lesssim \|\chi u_\omega\|_{H^{s}(\Omega)}.
    \end{equation}
    Then combining~\eqref{eq:fred_interior} and~\eqref{eq:fred_ends} yields the desired estimates. 
\end{proof}

Again, we need the analogue of Proposition~\ref{prop:SFE} with a derivative in the spectral parameter.
\begin{proposition}\label{prop:SFE_diff}
    Let $\omega$ and $u_\omega$ be the same as in Proposition~\ref{prop:SFE}. Then for any $\beta < -3/2$, $s \in \R$, and $N \in \NN$,
    \[\|\partial_\omega u_\omega \|_{\bar H^{s, \beta}(\Omega)} \le C \big(\|f\|_{H^{s + 1}(\Omega)} + \|u_\omega\|_{\bar H^{2, -N}(\Omega)} + \|\partial_\omega u_\omega\|_{\bar H^{2, -N}(\Omega)} \big).\]
\end{proposition}
\begin{proof}
    1. We first recover $\partial_\omega u_\omega$ away from the ends using $v_\omega$ away from the ends by using~\eqref{eq:single_layer}: differentiating this equation in $\omega$ yields
    \[\chi_\mb \partial_\omega u_\omega = \partial_\omega(E_\omega * (\indic_\omega[P(\omega), \chi_\mb] u_\omega)) + \partial_\omega E_\omega * f - (\partial_\omega S_\omega) \chi_\mb v_\omega + S_\omega \chi_\mb (\partial_\omega v_\omega).\]
    As in to Step 1 of Proposition~\ref{prop:SFE}, we let $\chi \in \CIc(\R)$ be such that $\chi = 1$ on $[-M, M]$ and $\supp \chi \subset (-\frac{3}{2} M, \frac{3}{2} M)$. Then using the mapping properties of Lemma~\ref{lem:S_mapping}, we find that
    \begin{equation*}\begin{split}
        \|\chi \partial_\omega u_\omega \|_{\bar H^s(\Omega)} \lesssim & \|\chi_{\mathrm{int}}\partial_\omega v_\omega\|_{H^{s - 1}(\partial \Omega)} + \|\chi_{\mathrm{int}} v_\omega \|_{H^{s}(\partial \Omega)} \\
        & + \|\partial_\omega u_\omega \|_{\bar H^{2, -N}(\Omega)} + \|u_\omega \|_{\bar H^{2, -N}(\Omega)} + \|f\|_{H^{s}(\Omega)}.
    \end{split}\end{equation*}
    Then applying Lemmas~\ref{lem:interior_prop} and~\ref{lem:interior_prop_diff}, we have
    \begin{equation*}\begin{split}
        \|\chi \partial_\omega u_\omega \|_{\bar H^s(\Omega)} \lesssim & \|\widetilde \Pi_+ \chi_{\mathrm{R}} \partial_\omega v_\omega \|_{H^{s - 1}(\partial \Omega)} + \|\widetilde \Pi_- \chi_{\mathrm{L}} \partial_{\omega} v_\omega \|_{H^{s - 1}(\partial \Omega)} \\
        & + \|\widetilde \Pi_+ \chi_{\mathrm{R}} v_\omega \|_{H^{s}(\partial \Omega)} + \|\widetilde \Pi_- \chi_{\mathrm{L}} v_\omega \|_{H^{s}(\partial \Omega)} \\
        & + \|u_\omega\|_{\bar H^{2, -N}(\Omega)} + \|\partial_\omega u_\omega\|_{\bar H^{2, -N}(\Omega)} + \|f\|_{H^{s+1}(\Omega)}.
    \end{split}\end{equation*}
 Using the high frequency decay estimates of Proposition~\ref{prop:high_decay} and~\ref{prop:high_decay_deriv} gives
    \[\|\chi \partial_\omega u_\omega \|_{H^s(\Omega)} \lesssim \|\partial_\omega u_\omega \|_{\bar H^{2, -N}(\Omega)} + \|u_\omega \|_{\bar H^{2, -N}(\Omega)} + \|f\|_{H^{s + 1}(\Omega)}.\]
    Finally, applying the low frequency decay estimates of Proposition~\ref{prop:low_decay_deriv} to $(1 - \chi) \partial_\omega u_\omega$ in the ends, we conclude that 
    \[\|\partial_\omega u_\omega \|_{\bar H^s(\Omega)} \lesssim \|f\|_{H^{s+ 1}(\Omega)} + \|u_\omega\|_{\bar H^{2, -N}(\Omega)} + \|\partial_\omega u_\omega\|_{\bar H^{2, -N}(\Omega)}\]
    as desired. 
\end{proof}

\begin{Remark}
    For Propositions~\ref{prop:SFE} and~\ref{prop:SFE_diff}, the Sobolev order needed on the right-hand-side for $f$ is not optimal. An improvement of one order can be obtained if one instead uses a sharper version of Lemma~\ref{lem:gw_bound}. However, it is simpler to assume $f \in \CIc(\Omega)$ so that the one extra derivative needed on $f$ does not ultimately affect the evolution problem. 
\end{Remark}

\section{Spectral and evolution problem}\label{sec:LAP}
With the global high frequency estimates of Proposition~\ref{prop:SFE} in place, we now use the uniqueness result in~\cite{LiWaWu:24} to establish the desired limiting absorption principle. In turn, the limiting absorption principle allows us to characterize the spectral measure near real values of the spectral parameter $\lambda$ such that $\Omega$ is $\lambda$-subcritical. This gives us a precise description of the long-time behavior of the evolution problem. As before, $\mathcal J \subset [0, 1]$ denotes an open interval such that $\Omega$ is $\lambda$-subcritical with respect to every $\lambda \in \mathcal J$.

\subsection{Limiting absorption principle}
Recall that \emph{incoming} and \emph{outgoing} solutions to the stationary equation are defined in Definition~\ref{def:io}
according to the Fourier modes present in their Neumann data in the ends.

\begin{proposition}\label{prop:LAP}
    Assume that $\omega_j \to \lambda \in \mathcal J$, $\Im \omega_j > 0$. Then for all $\beta < -1/2$ and $s \ge 1$, 
    \[u_{\omega_j} \to u_{\lambda + i0} \quad \text{in} \quad \bar H^{s, \beta}_0(\Omega) \quad \text{as} \quad j \to \infty\]
    where $u_{\lambda + i0}$ is the unique incoming solution to the stationary internal waves equation~\eqref{eq:SIW}. 
\end{proposition}
\begin{proof}
    1. It suffices to prove the proposition for $s > 2$. We first claim that $\|u_{\omega_j}\|_{\bar H^{s, \beta}(\Omega)}$ must be bounded. By way of contradiction, then, assume otherwise.  Upon passing to a subsequence, we may assume that $\|u_{\omega_j}\|_{\bar H^{s, \beta}(\Omega)} \to \infty$. Define the rescaling
    \[\tilde u_j = \frac{u_{\omega_j}}{\|u_{\omega_j}\|_{\bar H^{s, \beta}(\Omega)}}.\]
    Since $\bar H^{s, \beta}(\Omega) \hookrightarrow \bar H^{2, -N}(\Omega)$ compactly, we can further pass to a subsequence and assume that $\tilde u_j \to \tilde u_0$ in $\bar H^{2, -N}(\Omega)$. Observe that 
    \[P(\omega_j) \tilde u_j = \frac{f}{\|u_{\omega_j}\|_{\bar H^{s, \beta}(\Omega)}} \to 0\]
    in all weighted Sobolev norms.  Since $P(\omega_j) \to P(\lambda)$ as second-order differential operators, it follows that
    \begin{equation}\label{eq:pu0}
        P(\lambda) \tilde u_0 = 0.
    \end{equation}
    Now we apply the semi-Fredholm estimate of Proposition~\ref{prop:SFE} with $s' > s$ and $-\frac{1}{2} > \beta' > \beta$ to find that 
    \[\|\tilde u_j\|_{\bar H^{s', \beta'}(\Omega)} \le \|\tilde u_j\|_{\bar H^{2, -N}(\Omega)} + \frac{\|f\|_{H^{s'}(\Omega)}}{\|u_{\omega_j}\|_{\bar H^{s, \beta}(\Omega)}},\]
    which in particular implies that $\|\tilde u_j\|_{\bar H^{s', \beta'}(\Omega)}$ is bounded as $j \to \infty$. Again passing to a subsequence, we can assume that $\tilde u_j \to \tilde u_0$ in $\bar H^{s, \beta}(\Omega)$. Therefore, 
    \begin{equation}\label{eq:nontrivial}
        \|\tilde u_0\|_{\bar H^{s, \beta}(\Omega)} = 1.
    \end{equation}

    \noindent
    2. Next, let us show that $\tilde u_0$ is an incoming solution. Let $\tilde v_j$ and $\tilde v_0$ denote the Neumann data of $\tilde u_j$ and $u_0$ respectively.
    Employing the cutoffs $\chi_\pm$ from Proposition~\ref{prop:uniqueness},
    we see by continuity that $\widetilde \Pi_\pm \chi_\pm \tilde v_j \to \widetilde \Pi_{\pm} \chi_\pm \tilde v_0$ in $H^{s - 2, \beta}(\partial \Omega)$. Applying the high frequency decay estimate in Proposition~\ref{prop:high_decay} to the Neumann data $\tilde v_j$ with $s' > s - 2$ and $\alpha' > -1/2$, we also see that
    \[\|\widetilde \Pi_\pm \chi_\pm \tilde v_j \|_{\bar H^{s', \alpha'}(\partial \Omega)} \le \|\tilde u_j\|_{\bar H^{2, -N}} < C.\]
    Upon passing to a subsequence, we may assume that $\widetilde\Pi_\pm \chi_\pm \tilde v_j \to \widetilde\Pi_\pm \chi_\pm \tilde v_0$ in $\bar H^{s - 2, \alpha}(\Omega)$ for some $\alpha \in (-\frac{1}{2}, \alpha')$. Therefore, $\tilde u_0$ is a incoming solution to~\eqref{eq:pu0}, and it is nontrivial by~\eqref{eq:nontrivial}, which contradicts Proposition~\ref{prop:uniqueness}.

    \noindent
    3. So far, we have shown that $u_{\omega_j}$ is bounded in $\bar H^{s, \beta}_0(\Omega)$ for all $s > 2$ and $\beta < -1/2$. Then upon passing to a subsequence, we may assume that $u_{\omega_j} \to u_0$ (weakly) for some $u_0 \in \bar H^{s, \beta}_0(\Omega)$, and $u_0$ must satisfy
    \[P_\omega u_0 = f.\]
    Furthermore, by applying the high frequency decay estimate of Proposition~\ref{prop:high_decay} as in Step 2, we conclude that $u_0$ is incoming, so by Proposition~\ref{prop:uniqueness}, 
    \[u_0 = u_{\lambda + i0}.\]
    Therefore, all convergent subsequences of $u_{\omega_j}$ in $\bar H_0^{s, \beta}(\Omega)$ converge to $u_{\lambda + i0}$, which in turn implies that $u_{\omega_j} \to u_{\lambda + i0}$ in $\bar H_0^{s, \beta}(\Omega)$. 
\end{proof}
To access higher regularity of the spectral measure, we also need the following differentiated limiting absorption principle. 
\begin{proposition}\label{prop:LAP_diff}
    Let $u_{\lambda + i0}$ be the unique incoming solution to the stationary internal wave equation~\eqref{eq:SIW}. Then $(\lambda \mapsto u_{\lambda + i0}) \in C^1(\mathcal J; \bar H_0^{s, \beta}(\Omega))$ for all $\beta < -1/2$ and $s \ge 1$. In particular, if $\omega_j \to \lambda \in \mathcal J$, $\Im \omega_j > 0$, then 
    \[\partial_\omega u_{\omega_j} \to \partial_\lambda u_{\lambda + i0} \quad \text{in} \quad \bar H_0^{s, \beta}(\Omega) \quad \text{as} \quad j \to \infty.\] 
\end{proposition}
\begin{proof}
The proof is similar to that of Proposition~\ref{prop:LAP}.

    1. Let $\beta < -3/2$ and $s > 2$. By way of contradiction, assume that $\|u_\omega\|_{\bar H^{s, \beta}(\Omega)}$ is unbounded as $\epsilon \to 0$. Then pick a sequence $\omega_j$ so that $\|\partial_\omega u_{\omega_j}\|_{\bar H^{s, \beta}} \to \infty$. Again, define 
    \[\tilde u_j = \frac{\partial_\omega u_{\omega_j}}{\|\partial_\omega u_{\omega_j}\|_{\bar H^{s, \beta}(\Omega)}}.\]
    This is clearly bounded in $\bar H^{s, \beta}(\Omega)$, so upon passing to a subsequence, we may assume that it converges in $\bar H^{2, -N}(\Omega)$ to some $\tilde u_0 \in \bar H^{2, -N}(\Omega)$. Note that
    \[P(\omega_j) \tilde u_j = \frac{2 \omega \Delta u_{\omega_j}}{\|\partial_\omega u_{\omega_j}\|_{\bar H^{s, \beta}(\Omega)}}\]
    By Proposition~\ref{prop:LAP}, right-hand-side tends to zero in $\bar H^{s, \beta}$ for every $s$, and the left-hand-side converges in tempered distributions, hence
    \begin{equation}\label{eq:lim_free}
        P(\lambda) \tilde u_0 = 0.
    \end{equation}
    Let $s' > s$ and $\beta' \in (\beta, -3/2)$. By Proposition~\ref{prop:SFE_diff}, we also have 
    \[\|\tilde u_j \|_{\bar H^{s', \beta'}(\Omega)} \le C \left(\frac{\|f\|_{H^{s' + 1}(\Omega)} + \|u_{\omega_j} \|_{\bar H^{2, -N}(\Omega)}}{\|\partial_\omega u_{\omega_j}\|_{\bar H^{s', \beta}(\Omega)}} + \|\tilde u_{\omega_j} \|_{\bar H^{2, -N}(\Omega)} \right)< C, \]
    where we used the fact that $u_j$ is bounded in $\bar H^{2, -N}(\Omega)$ by Proposition~\ref{prop:LAP}, and $C$ may change from line to line. Therefore, $\tilde u_j \to \tilde u_0$ in $\bar H^{s, \beta}(\Omega)$, which in particular means that
    \[\|\tilde u_0\|_{\bar H^{s, \beta}(\Omega)} = 1.\]
    Next, we can estimate the Neumann data $\tilde v_j$ of $\tilde u_j$. It follows from Proposition~\ref{prop:high_decay_deriv} that 
    \[\|\widetilde \Pi_+ \chi_+ \tilde v_j\|_{\bar H^{s, \alpha}(\Omega)} < C\]
    for all $s, \alpha \in \R$, so $\|\widetilde \Pi_+ \chi_+ \tilde v_0\|_{\bar H^{s, \alpha}(\Omega)} < \infty$. Take $\alpha > 1/2$, which gives that $\tilde u_0$ is a nontrivial incoming solution to \eqref{eq:lim_free}, contradicting Proposition~\ref{prop:uniqueness}

    \noindent
    2. We have shown that the sequence $\partial_{\omega} u_{\omega_j}$ is bounded in $\bar H^{s, \beta}$ for every $s \in \R$ and $\beta < -3/2$. Fix $s > 2$ and $\beta < -3/2$, and choose a subsequence that converges to some limit $\dot u_{\lambda + i0} \in \bar H^{s, \beta}(\Omega)$. Note that the limit solves the equation
    \begin{equation}\label{eq:dotu}
        P(\lambda) (\dot u_{\lambda + i0}) = 2 \lambda \Delta u_{\lambda + i0}.
    \end{equation}
    By Proposition~\ref{prop:high_decay_deriv} it follows that $\dot u_{\lambda + i0}$ is incoming in the sense of Definition~\ref{def:io}. If a different limit exists for some other choice of a subsequence of $\partial_\omega u_{\omega_j}$, then the limit must also be incoming and satisfy~\eqref{eq:dotu}. Then the difference with $\dot u_{\lambda + i0}$ must also be incoming and solves the stationary internal wave equation with zero forcing. So it follows from Proposition~\ref{prop:uniqueness} that the whole sequence converges to the same limit, that is, 
    \[\partial_\omega u_{\omega_j} \to \dot u_{\lambda + i0} \quad \text{in} \quad \bar H^{s, \beta}(\Omega).\]
    Since we have convergence for any sequence of $\omega_j \to \lambda \in \mathcal J$, it follows that $\partial_{\lambda} u_{\lambda + i \epsilon} \to \dot u_{\lambda + i0}$ uniformly in $\bar H^{s, \beta}(\Omega)$ for $\lambda \in \mathcal J$. Therefore, $\dot u_{\lambda + i0} = \partial_\lambda u_{\lambda + i0}$. 
\end{proof}

Together with Stone's formula, Propositions~\ref{prop:LAP} and~\ref{prop:LAP_diff} immediately implies a slightly stronger version of Theorem~\ref{thm:spectral}. 
\begin{proof}[Proof of Theorem~\ref{thm:spectral}]
    Flipping the sign of $\epsilon$ in the proof of Proposition~\ref{prop:LAP}, we also see that if $\omega_j \to \lambda \in \mathcal J$, $\Im \omega_j < 0$, then 
    \[u_{\omega_j} \to u_{\lambda - i0} \quad \text{in} \quad \bar H^{s, \beta}_0(\Omega) \quad \text{as} \quad j \to \infty\]
    for every $s \ge 1$ and $\beta < -1/2$, where $u_{\lambda - i0}$ is the unique outgoing solution to the stationary internal waves equation~\eqref{eq:SIW}. 
    
    By Stone's formula, the spectral measure $E(z) \, \md z$ of $P$ is given for $z \in \mathcal J^2$ and $f\in C_c^\infty(\Omega; \R)$ by
    \begin{equation}\label{eq:stones}
        E(z)f = \frac{1}{2 \pi i} \Delta (u_{\sqrt{z} + i0} - u_{\sqrt{z} - i0}) \in C^1(\mathcal J^2; \bar H^{s, \beta}(\Omega)),
    \end{equation}
    for every $s \in \R$ and $\beta < -1/2$. Note that the $C^1$ regularity is a consequence of Proposition~\ref{prop:LAP_diff}.
\end{proof}

\subsection{Long-time evolution}
Finally, we return to the evolution problem. The following analysis is nearly identical to~\cite[\S7]{Li:24}, where we simply solve the evolution problem using the functional solution~\eqref{eq:functional_sol}. We will use the notation introduced in~\eqref{eq:functional_sol} below. 
\begin{proof}[Proof of Theorem~\ref{thm:tide}]
Since $\Omega$ being $\lambda$-subcritical is an open condition, there exists $\delta > 0$ so that $\Omega$ is $\lambda'$-subcritical for every $\lambda'^2 \in [\lambda^2 - 2\delta, \lambda^2 + 2 \delta]$. Fix a cutoff $\varphi \in \CIc(\R; [0, 1])$ so that $\varphi = 1$ on $[\lambda^2 - \delta, \lambda^2 + \delta]$ and $\supp \varphi \subset (\lambda^2 - 2 \delta, \lambda^2 + 2 \delta)$. We split the solution $u(t) = \Delta_\Omega^{-1} w(t)$ into two pieces by
\begin{equation*}
    \begin{gathered}
        u(t) = \Re\big(e^{i \lambda t} (u_1(t) + r_1(t))\big) \quad \text{where} \\
        u_1(t) \coloneqq \Delta_\Omega^{-1} \varphi(P) \mathbf W_{t, \lambda}(P) f, \qquad r_1(t) \coloneqq \Delta_\Omega^{-1}(I - \varphi(P)) \mathbf W_{t, \lambda}(P)f.
    \end{gathered}
\end{equation*}
Observe that 
\[\big|\Re\big(e^{i \lambda t} \mathbf W_{t, \lambda}(z) (1 - \varphi(z))\big)\big| = \frac{|(\cos(t \sqrt{z}) - \cos(t \lambda))|(1 - \varphi(z))}{|\lambda^2 - z|} \le \frac{2}{\lambda \delta}.\]
Therefore, it follows from the spectral theorem that 
\[\|\Re(e^{i \lambda t}r_1(t))\|_{\bar H^{1}_0(\Omega)} \le \frac{2}{\lambda \delta} \|f\|_{ H^{-1}(\Omega)}\]
for all $t \ge 0$. 

By~\eqref{eq:stones}, there exists $\nu(z) \coloneqq \Delta_{\Omega}^{-1}E(z)f \in C^1([\lambda^2 - 2\delta, \lambda^2 + 2 \delta]; \bar H^{s, \beta}_0(\Omega))$ with $s\geq 1$ and $\beta<\frac12$, such that 
\[u_1(t) = \int_\R \varphi(z) \mathbf W_{t,\lambda}(z) \nu(z) \, \md z.\]
Recalling the definition of $\mathbf W_{t,\lambda}(z)$ in \eqref{eq:functional_sol}, we compute
\[\begin{split}
    u_1(t) = & \frac{1}{2i}\int_0^t\int_\R \left( e^{is(\sqrt z-\lambda)} - e^{-is(\sqrt z+\lambda)} \right)\frac{1}{\sqrt z} \varphi(z)\nu(z) \, \md z\,\md s \\
    = & -i\int_0^t \int_\R \left(  e^{is(\zeta-\lambda)} - e^{- is (\zeta + \lambda)} \right)\varphi(\zeta^2)\nu(\zeta^2) \, \md \zeta \, \md s \\
    = & -i \int_{-t}^0 \widehat q_+(s) \,\md s + i \int_{0}^t \widehat q_- (s) \, \md s
\end{split}\]
where $q_\pm(\zeta)\coloneqq \varphi((\zeta\pm\lambda)^2)\nu((\zeta \pm \lambda)^2)$, hence $q_\pm\in C^1_c(\RR; \bar H^{s,\beta}_0(\Omega))$. Consequently,
\[ \int_\R \|\widehat q_\pm(s)\|_{\bar H^{s,\beta}}\,\md s\leq C \left( \int_\R \langle s \rangle^{1+2\alpha_0}\|\widehat q_\pm(s)\|^2_{\bar H^{s,\beta}} \,\md s  \right)^{\frac12} = C\|q_\pm\|_{H^{\frac12+\alpha_0}(\RR; \bar H^{s,\beta})}<\infty \]
for any $0<\alpha_0\leq \frac12$.
Hence by the Dominated Convergence Theorem, we conclude that in $\bar H^{s,\beta}(\Omega)$, as $t\to +\infty$,
\[\begin{split}
    u_1(t) 
    \to & -i \int_{-\infty}^{0} \widehat q_+(s)\,\md s + i \int_0^{+\infty} \widehat q_-(s)\,\md s \\
    = & \sum_{\pm}\int_\R (\zeta\pm i0)^{-1} q_\pm(\zeta)\,\md \zeta 
    = \Delta_{\Omega}^{-1}\varphi(P)(P-\lambda^2+i0)^{-1}f.
\end{split}\]
Thus we can write 
\[\begin{gathered} 
u_1(t) = \Delta_{\Omega}^{-1}(P-\lambda^2+i0)^{-1}f + r_2(t) + \tilde{\mathcal E}(t)
\end{gathered}\]
with $\|\tilde{\mathcal E}(t)\|_{\bar H^{s,\beta}(\Omega)}\to 0$ and 
\[ r_2(t)\coloneqq \Delta_{\Omega}^{-1}(P-\lambda^2)^{-1}(\varphi(P)-1)f, \ \|r_2(t)\|_{\bar H^1_0(\Omega)}\leq C\|f\|_{H^{-1}(\Omega)}. \]

Theorem~\ref{thm:tide} then follows upon setting 
\[r(t) = \Re (e^{i \lambda t} (r_1(t) + r_2(t)), \quad \mathcal E(t) = \Re(e^{i \lambda t} \tilde{\mathcal{E}}(t)).\]
In addition, we also see that $u^+ = \Delta_\Omega^{-1} (P - \lambda^2 + i0)^{-1} f = u_{\lambda - i0}$ is the unique outgoing solution to the stationary internal wave equation~\eqref{eq:SIW}.
\end{proof}

\bibliography{evolution}
\bibliographystyle{alpha}

\end{document}